\documentclass[10pt]{amsart}
\usepackage{amssymb}
\usepackage[varbb, varg, smallerops]{newtxmath}
\usepackage[scaled=0.87]{inconsolata} 
\usepackage[T1]{fontenc}
\usepackage{amsfonts}
\usepackage{comment}
\usepackage[a4paper,margin=1.1in]{geometry} 
\usepackage{tikz}
\usepackage{tikz-cd}
\usetikzlibrary{arrows}
\tikzcdset{arrow style=tikz,
diagrams={>={Stealth[round,length=4pt,width=6pt,inset=3pt]}}}
\usepackage{graphicx}
\usepackage[all]{xy}
\makeatletter
\newcommand*{\tightdisplaymath}{\abovedisplayskip\z@\belowdisplayskip\z@}

\newcommand{\vers}{Tensor product of $\Ain$Categories}
\title[\vers]{Tensor product of $\Ain$Categories}

\usepackage{hyperref}
\hypersetup{colorlinks=false}

\author{Mattia Ornaghi}
\address{\parbox{0.9\textwidth}{Universit\`a degli Studi di Milano\\
Dipartimento di Matematica\\
Via Cesare Saldini 50, 20133 Milano, Italy}}
\email{mattia12.ornaghi@gmail.com}

\theoremstyle{definition}
\newtheorem{defn}{Definition}[section]

\newtheorem{thm}{Theorem}[section]

\newtheorem{lem}[thm]{Lemma}

\newtheorem{cor}[thm]{Corollary}

\newtheorem*{namedthm}{Theorem}

\theoremstyle{remark}
\newtheorem{rem}{Remark}[section]
\newtheorem{exmp}{Example}[section]
\newtheorem{notat}{Notation}[section]

\newcommand{\Ain}{\mbox{A$_{\infty}$}}
\newcommand{\id}{\mbox{id}}

\newcommand{\aCat}{\mbox{A$_{\infty}$Cat}}

\newcommand{\DgCat}{\mbox{DGCat}}

\newcommand{\Rep}{\underline{\mbox{Rep}}}
\newcommand{\REP}{\mbox{Rep}}
\newcommand{\Id}{\mbox{Id}}

\newcommand{\Un}{{U}^{\tiny\mbox{n}}}

\newcommand{\Hom}{\mbox{Hom}}
\newcommand{\HomA}{\mbox{Hom}_{\mathscr{A}}}

\newcommand{\A}{\mathscr{A}}

\newcommand{\B}{\mathscr{B}}

\newcommand{\Ho}{\mbox{Ho}}
\newcommand{\Fun}{\mbox{Fun}}

\newcommand{\G}{\mathscr{G}}
\newcommand{\Q}{\mathscr{Q}}
\newcommand{\F}{\mathscr{F}}
\newcommand{\C}{\mathscr{C}}

\newcommand{\Z}{\mathbb{Z}}


\thanks{The author was supported by the research project FARE 2018 HighCaSt (grant number R18YA3ESPJ) and by ERC Advanced
Grant - CUP: G43C23001750006 - HE-ERC23ANEEM-01}

\begin{document}

\date{\today}

\maketitle
\begin{abstract}
In this paper we define the tensor product of two $\Ain$-categories and two $\Ain$-functors.\ 
This tensor product makes the category of $\Ain$-categories symmetric monoidal (up to homotopy), and 
the category $\aCat^u/\approx$ a closed symmetric monoidal category.\
Moreover, we define the derived tensor product making $\Ho(\aCat)$, the homotopy category of the $\Ain$-categories, 
a closed symmetric monoidal category.\ We provide also an explicit description of the internal homs in terms of $\Ain$-functors.
\end{abstract}

\setcounter{tocdepth}{1}

\tableofcontents
\section{Introduction}

We fix a commutative ring $R$.\
Let $A,B$ be two DG $R$-quivers, the tensor product $A\otimes B$ is a DG quiver defined as follows:\
\begin{itemize} 
\item[1.] The objects are the pair $\mbox{Ob}(A)\times\mbox{Ob}(B)$,
\item[2.] the morphisms are 
\begin{align}\label{Tensorello}
({A\otimes B})^n\big((x_1,y_1),(x_2,y_2)\big):=\displaystyle\bigoplus_{i+j=n} A^i(x_1,x_2)\otimes B^j(y_1,y_2),
\end{align}
where $x_1,x_2\in A$ and $y_1,y_2\in B$.\
\end{itemize}
Fixed a morphism $a\otimes b\in A\otimes B$, the differential is given by
\begin{align}\label{diiffo}
d_{A\otimes B}(a\otimes b):=d_{A}(a)\otimes b + (-1)^{\tiny\mbox{deg}(a)}a\otimes d_{B}(b).
\end{align}
If $A$ and $B$ are two DG-categories (not necessarily unital) then the formula: 
\begin{align}\label{compoz}
(a_2\otimes b_2)\cdot (a_1\otimes b_1):=a_2\cdot a_1\otimes b_2\cdot b_1.
\end{align}
makes $A\otimes B$ a DG-category.\ Moreover, taking two DG-functors $F:A\to B$ and $G:A'\to B'$, it is easy to define the tensor product $F\otimes G:A\otimes A'\to B\otimes B'$ as follows:
\begin{align}\label{functoro}
(F\otimes G) (a\otimes a' ):=F(a)\otimes G(a').
\end{align}
It is easy to prove that the category of DG-categories, with the tensor product defined above, is a closed symmetric monoidal category.\

Things gets more complicated if $A$ is a DG-category and $B$ is an $\Ain$-category.\ 
If $n\ge2$, we can extend formulas ($\ref{diiffo}$) and $(\ref{compoz})$ as follows:
\begin{align}\label{zorku}
m^{n}_{A\otimes B}(a_n\otimes b_n,...,a_1\otimes b_1):=\pm\ a_n\cdot... \cdot a_1\otimes m^n(b_n,...,b_1).
\end{align}
It is not hard to prove that $(\ref{zorku})$ makes $A\otimes B$ an $\Ain$-category, see \cite[(3.)]{GZ}, \cite[Appendix C]{BLM} or \cite[Remark 1.11]{Sei}.\ 
On the other hand, if $A$ and $B$ are both $\Ain$-categories, formula $(\ref{zorku})$ does not make sense, since $A$ is not associative.\ 
An explicit formula for $m^n_{A\otimes B}$ is highly non trivial (see \cite[\S5.2]{KS}).\\

Let us give a brief historical background.\
The question about the tensor product of two $\Ain$-categories came up in the context of Open String Theory.\ 
In 1997, indeed, Gaberdiel and Zwiebach were the first who faced this problem.\ They managed to 
give an explicit description of the tensor product of an $\Ain$-algebra with a DG one \cite[3.14]{GZ} and of the first homotopy of the tensor product
algebra (i.e. $m^3_{A\otimes B}$, \cite[3.13]{GZ}).\ Namely they provided formulas (\ref{zorku}) and (\ref{mtre}).\\
The first who provided a complete, formula for the $\Ain$structure on the tensor product were Saneblidze and Umble \cite{SU}, in 2004.\
A few years later, in 2006, Markl and Shnider \cite{MS} gave a more conceptual proof.\ Here is a theoretical explanation of their strategy:\
First, we denote by $\mathcal{A}_{\infty}$ the operad describing the $\Ain$algebras (called the Stasheff polytope or associahedron).\ Namely, every $\Ain$algebra is an algebra over the operad $\mathcal{A}_{\infty}$.\ 
Roughly speaking, every diagonal 
$$\Delta:\mathcal{A}_{\infty}\to \mathcal{A}_{\infty}\otimes \mathcal{A}_{\infty},$$ 
is equivalent to an $\Ain$-structure on the tensor product of the underlying DG quivers \cite[pp.45]{MS}.\\
Their $\Ain$-structure on the tensor was based on the Boardman-Vogt cubical decomposition of the Stasheff polytope.\ 
Another $\Ain$-structure was provided by Loday, in 2007 \cite{Lod}.\ 
His strategy is different, since he used the simplicial (not the cubical) decomposition of the associahedron.\ 
Saneblidze-Umble-Markl-Shnider and Loday constructions are quasi-isomorphic but he conjectured that they are actually the same.\ It was verified by Loday up to $n\le 6$.\ 
More recently contributions on the decomposition of the Stasheff polytope were done by Laplante-Anfossi, Mazuir \cite{LAM} and Masuda, Thomas, Tonks and Vallette \cite{MTTV}.\ Very recently, in 2023, Saneblidze and Umble (cf. \cite{SaU}) proved that the structure they found in \cite{SU}, the one of Markl-Shnider, and the aforementioned ones, agree.\

On the other hand, it is not strictly necessary to invoke the abstract language of the operads to find a good notion of tensor product for two $\Ain$-categories.\ 
Indeed, in \cite[\S5.2]{KS}, Kontsevich and Soibelman defined the tensor product of two $\Ain$-categories simply as the (DG )tensor product of the DG-categories of their modules:
\begin{align}\label{KontSob}
\A\otimes\B:=\mbox{End}(\A)\otimes\mbox{End}(\B).
\end{align}
The drawback of Definition $(\ref{KontSob})$ is that it is not an $\Ain$structure on the underlying DG-quivers of $A$ and $B$, but it has a different underlying DG-quiver.\
It was Amorim, in 2016 \cite{Amo}, who makes explicit their idea, transferring the DG-structure of $\mbox{End}(\A)\otimes\mbox{End}(\B)$ on the DG-quiver $A\otimes B$, using the Homology Perturbation Theory.\ In Section \ref{TENSHPL} we provide an extensive discussion about this strategy.\\

It important to say that what is really difficult in the context of $\Ain$-categories is the definition of tensor between two $\Ain$-functors.\
Namely, given two $\Ain$-functors $\F:\A\to\B$ and $\G:\A'\to\B'$ we would like to define the $\Ain$-functor
\begin{align}\label{compolo}
\F\otimes\G:\A\otimes\B\to\A'\otimes\B'
\end{align}
extending the tensor product of DG-functors in equation $(\ref{functoro})$ (cf. \cite[pp. 4]{MS}).\\
On the other hand, there are a few of definitions of tensor product of $\Ain$-functors (see \cite[4.4.3]{LAM} and \cite[Definition 3.9]{Amo}) unfortunately 
none of them has the property:
\begin{align}\label{formulaz}
(\F\otimes\G)\cdot(\F'\otimes\G')=(\F\cdot\F')\otimes(\G\cdot\G')
\end{align}
but 
it seems that 
$(\ref{formulaz})$ holds up to an appropriate homotopy relation, see \cite[Definition 3.9]{Amo} and \cite[4.4.4]{LAM}.\
The aim of this work start from this: we want to define a tensor product for $\Ain$-categories and for $\Ain$-functors, giving rise to an homotopical symmetric monoidal structure on $\aCat$.

\subsection{Statement of the results}
Let us make precise the results of this paper.\
First, we should say that there are different notions of unitary $\Ain$-categories: the strictly unital $\Ain$-categories (su), unital $\Ain$-categories (u), cohomology unital $\Ain$-categories (cu) and non unital $\Ain$-categories (nu).\ 
We will treat all of them.\ To state the next results, we will use the symbol $\star$ where 
$$\star\in\mathcal{f}\mbox{non unital, cohomological unital, unital, strictly unital}\mathcal{g}.$$

Given two $\star$ $\Ain$-categories $\A$ and $\B$, we define the tensor product $\A\otimes\B$ (Definition \ref{tenso}).\ 
Its existence follows from Theorem \ref{enzone}).\ The tensor product $\A\otimes\B$ is a $\star$ $\Ain$-category extending the DG-one.\
In particular it has the following properties:
\begin{itemize}
\item[1.] The objects of $\A\otimes\B$ are the pair $\mbox{Ob}(\A)\times \mbox{Ob}(\B)$. 
\item[2.] The morphism of $\A\otimes\B$ are the ones of equation (\ref{Tensorello}). 
\item[3.] $m^1_{\A\otimes\B}$ is the differential in equation (\ref{diiffo}).
\item[4.] $m^2_{\A\otimes\B}$ is given by the composition in equation (\ref{compoz}).
\item[5.] If $\A$ is a DG-category then $\A\otimes\B$ is given by equation (\ref{zorku}). 
\item[6.] If $\A$ and $\B$ are DG-categories then $\A\otimes\B$ is the usual tensor product.
\end{itemize}
Moreover, given two $\star$ $\Ain$-functors $\F:\A\to\A'$ and $\G:\B\to\B'$ we define the $\star$ $\Ain$-functor 
\begin{align*}
\F\otimes\G:\A\otimes\B\to\A'\otimes\B'
\end{align*}
which extends the usual tensor product of DG-functors.\ 
To state the (main) properties of such a tensor product of $\Ain$-functors we need a few of preliminaries.\\

Given an $\Ain$-functor $\F$ we say that $\F$ is an \emph{equivalence}  if there exists $\G$ such that $\F\cdot \G=\Id_{\B}$ and $\G\cdot \F=\Id_{\A}$.\ It is important to say that in the framework of $\Ain$-categories there are weaker notions of {equivalence}.\

Recalling that, given an $\Ain$-category $\A$, its cohomology $H(\A)$ is a $R$-linear category, 
we say that an $\Ain$-functor is a \emph{quasi-equivalence} if it induces an equivalence of their cohomology categories.\ Clearly 
\begin{center}
$\F$ is \emph{equivalence} $\Rightarrow$ $\F$ is \emph{quasi-equivalence}.
\end{center}
We call homotopy category of the $\Ain$-categories the Gabriel-Zisman \cite{GaZi} localization of the category of $\Ain$-categories with respect to the quasi-equivalences, we denote it by $\Ho(\aCat)$.\\ 
Fixed two $\Ain$-categories $\A$, $\B$ the hom-space 
\begin{align*}
{\aCat}(\A,\B)
\end{align*}
has a natural enrichments in $\aCat$.\ Namely, the set of $\Ain$-functors has an $\Ain$-structure (denoted by $\mathfrak{M}^n$) \cite{Lyu}, whose morphisms are the prenatural transformations.\
We denote this $\Ain$-category by $\Fun_{\infty}(\A,\B)$.\\ 
We say that two $\Ain$-functors $\F$ and $\G$ are \emph{homotopic} (in formula $\F\sim\G$) if there exists a prenatural transformation $T$ such that $\F-\G=\mathfrak{M}^1(T)$.\ We denote by $\aCat/_\sim$ the category of $\Ain$-categories whose morphisms are the $\Ain$-functors quotient by the equivalence relation $\sim$.\\
To conclude, we say that $\F$ and $\G$ are \emph{weakly equivalent} (in formula $\F\approx\G$) if they are isomorphic as objects in $H^0(\Fun_{\infty}(\A,\B))$.\ We have:
\begin{center}
\emph{homotopic} ($\F\sim\G$) $\Rightarrow$ \emph{weakly-equivalent} ($\F\approx\G$) $\Rightarrow$ \emph{quasi-equivalent} ($[\F]=[\G]$).
\end{center}

Coming back to our tensor product of $\Ain$-functors, we have the following property:

\begin{namedthm}[\ref{trunchezfol}]
Given four $\star$ $\Ain$-functors $\F:\A\to\B$, $\G:\A'\to\B'$, $\F':\C\to\A$ and $\G':\C'\to\A'$ we have:
\begin{align}\label{homuz}
(\F\otimes\G)\cdot (\F'\otimes\G')\sim (\F\cdot\F'\otimes\G\cdot\G').
\end{align}
\end{namedthm}
In particular, we have:
\begin{align}\label{torquez}
[(\F\otimes\G) \cdot (\F'\otimes\G')]= [\F\cdot\F'\otimes\G\cdot\G']
\end{align}
in $\Ho(\aCat)$.\\
\\
Our proof of Theorem \ref{trunchezfol} is explicit.\ Namely we provide a prenatural trasformation making the $\Ain$-functors (the composition of the tensors and the tensor of the compositions) in equation (\ref{homuz}) homotopic.\
Moreover we prove that, if $\F\sim\G$ and $\F'\sim\G'$, then $(\F\otimes\F')\sim(\G\otimes\G')$.\ 
This provides a bifunctor 
\begin{align*}
\otimes:\aCat^{\star}/_\sim\mbox{ $\times$ }\aCat^{\star}/_\sim \to \aCat^{\star}/_\sim
\end{align*}
see Theorem \ref{gronzulcore}.\ 
In particular the category $(\aCat^{\star},\otimes,R)$ is symmetric monoidal up to homotopy.\ 
We prove the following results:

\begin{namedthm}[\ref{symcat}]
The category $(\aCat^{\star}/_\sim,\otimes,R)$ is symmetric monoidal.\
\end{namedthm}

\begin{namedthm}[\ref{closedsm}]
The category $(\aCat^{u}/_\approx,\otimes,R)$ is closed symmetric monoidal.\ 
Given two unital $\Ain$-categories $\A$ and $\B$, the internal hom $[\B,\C]$ can be described as $\Fun^{u}(\B,\C)$.\ 
\end{namedthm}
Here $\Fun^{u}(\B,\C)$ denotes the unital $\Ain$-category of the unital $\Ain$-functors with source $\A$ and target $\B$, whose objects are the prenatural transformations.\\

We point out that Theorem \ref{symcat} and \ref{closedsm} follow from the fact that $\sim$ and $\approx$ are "stable" with respect to the DG-tensor product (see Lemmas \ref{stupidone} and \ref{stupidonzio}).\ Namely, give three DG-functors $\mathsf{F}$, $\mathsf{G}$ and $\mathsf{H}$ then
$$\mathsf{F}\sim \mathsf{G}\Rightarrow \mathsf{F}\otimes\mathsf{H}\sim \mathsf{G}\otimes\mathsf{H},$$
$$\mathsf{F}\approx \mathsf{G}\Rightarrow \mathsf{F}\otimes\mathsf{H}\approx \mathsf{G}\otimes\mathsf{H}.$$
Note that Lemmas \ref{stupidone} and \ref{stupidonzio} are new in the literature.\\

To conclude, motivated by (\ref{torquez}), we define the \emph{derived tensor product} of two $\star$ $\Ain$-categories as follows:
\begin{align*}
\A\otimes^{\mathbb{L}}\B :=\A\otimes\B^{\tiny\mbox{hp}}
\end{align*}
where $\B^{\tiny\mbox{hp}}$ denotes a homotopically projective resolution of $\B$.\ 
We recall that the existence of a (functorial up to $\approx$) homotopically projective resolution of $\B$ is guaranteed by \cite{Orn2}.\\

The derived tensor product $\otimes^{\mathbb{L}}$ provides a symmetric monoidal structure on $\Ho(\aCat^{\star})$.\ We prove that it is closed and we describe the internal homs in terms of $\Ain$-functors \`a la Kontsevich-Keller \cite[Section 4.3]{Kel}.\
In particular we prove:
\begin{namedthm}[\ref{INTERNALHOM}]
Given $\A_1,\A_2,\A_3\in\star$, there exists a natural bijection of sets
\begin{align*}
{\mbox{Ho}(\star)}(\A_1\otimes^{\mathbb{L}}\A_2,\A_3)\simeq{\mbox{Ho}(\star)}(\A_1, \Fun^u_{\infty}(\A^{\tiny\mbox{hp}}_2,\A_3)),
\end{align*}
where $\star\in\mathcal{f}\DgCat^u,\aCat^u,\aCat\mathcal{g}$.\\
Moreover, if $\A_1,\A_2,\A_3\in\aCat$ and $\A_1\otimes^{\mathbb{L}}\A_2:=\A_1\otimes\tilde{\A}_2$, such that $\tilde{\A}_2$ is h-projective with split unit, then 
\begin{align*}
{\mbox{Ho}(\aCat)}(\A_1\otimes^{\mathbb{L}}\A_2,\A_3)\simeq{\mbox{Ho}(\aCat)}(\A_1, \Fun_{\infty}(\tilde{\A}_2,\A_3)).
\end{align*}
\end{namedthm}

\subsection*{Recent related works}

Some of the aforementioned works on tensor product of $\Ain$-categories works only over a fields (or field of characteristic zero) and for $\Ain$-algebras (namely categories with one object).\ We don't have any restrictions in this sense.\

Theorem \ref{INTERNALHOM} extends to the $\Ain$-categories and unital DG-categories, the description of the internal homs in $\Ho(\DgCat)$ proved by the author together with Alberto Canonaco and Paolo Stellari, see \cite{COS2}.\  
This is a non trivial generalization, for several technical reasons.

It is worth to say that in \cite{Tan}, making use of the following result: 
\begin{namedthm}[\cite{COS2}] 
We have an equivalence of $\infty$-categories:
\begin{align*}
\Ho_{\infty}(\aCat)\simeq \Ho_{\infty}(\DgCat),
\end{align*}
provided by the following DK-adjunction of categories:
\begin{align*}
U:(\aCat,\mathcal{W}^{\tiny\Ain}) \rightleftharpoons(\DgCat,\mathcal{W}^{\tiny\mbox{DG}}):i
\end{align*}
Where $\mathcal{W}^{\tiny\Ain}$ and $\mathcal{W}^{\tiny\mbox{DG}}$ denote the sets of quasi-equivalences.
\end{namedthm}
Tanaka 
induces a closed symmetric monoidal structure on the $\infty$-category of $\Ain$-categories.\
Moreover, using smartly some of results in \cite{COS2} he gives a description of the internal homs in terms of nerve of $\Ain$-functors
\cite[Theorem 1.5]{Tan}.\
The tensor product he induced in the $\infty$-category of the $\Ain$-categories, makes sense only at the $\infty$-level.\
On the other hand, given two $\Ain$-categories $\A$ and $\B$, the underlying DG-quiver of the tensor product (induced in this way) is not the $\A\otimes\B$ of formula (\ref{Tensorello}).\
Note also that, his description of internal homs is not natural, see \cite[Remark 3.38]{Tan}.\ 



\subsection*{Acknowledgements}

I thank Francesco Genovese, Alberto Canonaco and Paolo Stellari for all the
valuable conversations I benefitted during the writing of this paper.\ 
I am very grateful to Volodymyr Lyubashenko and Paul Seidel for the useful clarifications about their results.

\newpage

\section{Background}

In this section we briefly recall a few of basic definitions in the context of DG and $\Ain$-categories.\ 
We fix a commutative ring $R$.\ All our categories will be linear over $R$ and small.

\begin{defn}[Graded Quiver]
A \emph{graded quiver} $\textsf{Q}$ is given by: a set of objects and, fixed two objects $x$ and $y$ in $\textsf{Q}$, a graded $R$-module $\mathsf{Q}(x,y)$.
\end{defn}

\begin{defn}[DG-Quiver]
A \emph{DG-quiver} $\textsf{Q}$ is given by: a set of objects and, fixed two objects $x$ and $y$ in $\textsf{Q}$, a differential graded $R$-module $\mathsf{Q}(x,y)$.
\end{defn}

\begin{defn}[$\Ain$-structure]
Let $\mathsf{Q}$ be a graded quiver, an \emph{$\Ain$-structure} on $\mathsf{Q}$ is a sequence of multilinear maps:
\begin{align*}
m^{n}:\mathsf{Q}(x_{n-1},x_n)\otimes...\otimes\mathsf{Q}(x_0,x_1)&\to \mathsf{Q}(x_0,x_n)[2-n],
\end{align*}
satisfying:
\begin{align*}
\displaystyle\sum^{d}_{m=1}\sum^{d-m}_{n=0}(-1)^{\tiny\mbox{deg}(q_1)+...+\tiny\mbox{deg}(q_n)-n} m^{d-m+1}(q_d,...,q_{n+m+1},m^m(q_{n+m},...,q_{n+1}),q_n,...,q_1)=0.
\end{align*}
For every sequence of objects $x_0,...x_n\in\mathsf{Q}$ and $n\ge1$
\end{defn}

A non unital \emph{$\Ain$-category} is: a graded quiver $\mathsf{Q}$ and an $\Ain$-structure on $\mathsf{Q}$.\ 
A \emph{non unital DG-category} is an $\Ain$-category whose $\Ain$structure is such that $m^{n\ge3}=0$.
Given a non unital $\Ain$-category $\A$ we denote by $|\A|$ the underlying DG-quiver $(\A,m^1_{\A})$.

\begin{defn}[Non unital $\Ain$-functor]
Let $\A$ and $\B$ be two non unital $\Ain$-categories, a \emph{non unital $\Ain$-functor} $\F:\A\to\B$ is:
\begin{itemize}
\item[1.] A morphism of sets $\F^0:\mbox{Obj}(\A)\to\mbox{Obj}(\B)$.
\item[2.] Given $n\ge 1$, multilinear maps:
\begin{align*}
\F^{n}:\A(x_{n-1},x_n)\otimes...\otimes\A(x_0,x_1)&\to \B(x_0,x_n)[1-n],
\end{align*}
satisfying:
\begin{align*}
\displaystyle\sum^{d}_{m=1}&\sum^{d-m}_{n=0}m^{r}(\F^{s_r}(a_d,...,a_{n+m+1}),...,\F^{s_1}(a_n,...,a_1))=\\
&=\displaystyle\sum_{m,n}(-1)^{\tiny\mbox{deg}(a_1)+...+\tiny\mbox{deg}(a_n)-n}\F^{d-m+1}(a_d,...,a_{n+m+1},m^m(a_{m+n},...,a_{n+1}),a_{n},...,a_1).
\end{align*}
\end{itemize}
\end{defn}

A \emph{strict $\Ain$-functor} $\F$ is an $\Ain$-functor such that $\F^{n\ge2}=0$.

\begin{exmp}
Let $\A$ be a non unital $\Ain$-category, the identity functor $\Id_{\A}$ defined as $\Id^0_{\A}(x)=x$, $\Id^1(f)=f$ and $\Id^{n\ge2}=0$ is a (strict) $\Ain$-functor.  
\end{exmp}

\begin{defn}[Composition of $\Ain$-functors]
Let $\A$, $\B$ and $\mathscr{C}$ be three non unital $\Ain$-categories, and $\mathscr{F}:\A\to\B$ and $\mathscr{G}:\B\to\mathscr{C}$ two $\Ain$-functors we have an $\Ain$-functor $\G\cdot\F:\A\to\mathscr{C}$ defined as follows:
\begin{align}\label{compo}
(\G\cdot\F)^d(q_d,...,q_1):=\displaystyle\sum_{r}\sum_{s_1,...,s_r}\G^{r}(\F^{s_r}(q_d,...,q_{d-s_r+1}),...,\F^{s_1}(q_{s_1},...,q_1))
\end{align}
\end{defn}

The set of $\Ain$-categories with $\Ain$-functors form a (1-)category that we denote by $\aCat^{\tiny\mbox{nu}}$.\ 
In the same vein, the set of non unital DG-categories with $\Ain$-functors is denoted by $\aCat^{\tiny\mbox{nu}}_{\tiny\DgCat^{\tiny\mbox{nu}}}$ 
and the set of non unital DG-categories with strict $\Ain$-functors is denoted by $\DgCat^{\tiny\mbox{nu}}$.\\
\\
Given an $\Ain$-category $\A$ we can take the \emph{cohomology category} $\Ho(\A)$.\ 
This is a non unital graded category whose objects are the same of $\A$ and the morphisms are defined as follows:
\begin{align*}
{\Ho(\A)}(x,y):=\displaystyle\bigoplus_{i\in\mathbb{\Z}} H^i({\A}(x,y)),
\end{align*}
for every $x,y\in\A$.

\subsection{Different notions of unital $\Ain$-category}\label{nice}
In general, given a non unital $\Ain$-category its cohomology category is not unital category.\ 
In the framework of $\Ain$-category, we have different notions of \emph{unit}.
\begin{defn}[Strictly unital]\label{strictlyunital}
An $\Ain$-category $\mathscr{A}$ is \emph{strictly unital} if for all $x\in\mathscr{A}$ there exists a closed morphism {$e_x\in\A^0(x,x)$} such that:
\begin{itemize}
\item[i)] {$m^2_{\A}(f,e_x)=(-1)^{\tiny\mbox{deg($f$)}}\cdot m^2(e_x,f)=f$},
\item[ii)] $m^n(f_{n-1},...,f_{k},e_x,f_{k-1},...,f_1)=0$
\end{itemize}
for all morphisms $f, f_1,...,f_{i-1}$ and for all $n>2$ and $1\le k\le n$.
\end{defn}

Moreover, if $\A$ a strictly unital $\Ain$-category we say that $\A$ has a \emph{split} unit if, for every $x\in\A$, the unit $R$ is such that $R\cdot e_x \cong R$ and the following short exact sequence 
\begin{align*}
0\to R\cdot e_x \to {\A}(x,x) \to {\A}(x,x)/ R\cdot e_x\to 0
\end{align*}
splits in the category of graded $R$-modules.

\begin{rem}
Any morphism $e_x$ satisfying condition i) must be closed and of degree zero since $m^2(e_x,e_x)=e_x$.\ We have: 
\begin{align*}
m^1(e_x)=m^1(m^2(e_x,e_x))=m^2(m^1(e_x),e_x)-m^2(e_x,m^1(e_x))=0.
\end{align*}
\end{rem}

\begin{defn}[Cohomological unital]
An $\Ain$-category $\mathscr{A}$ is \emph{cohomological unital} if its cohomology category $H^0(\mathscr{A})$ is a category.
\end{defn}

\begin{rem}
By condition ii) we have that a strict unit $e_x$ (if it exists) is unique.\ 
On the other hand a cohomological unit $e_x$ it is unique up to $m^1$.
\end{rem}

\begin{defn}[Unital]
An {$\Ain$}-category without strictly unit $\A$ is \emph{unital} if, for every $x,y\in\A$
and any representative $e_x$ of the unit of $x$ in $H(\A)$, there exists: \emph{$\mathcal{H}:\A(x,y)\to\A(x,y)[-1]$} such that 
\begin{align}\label{unitrap}
\Id_{\A(x,y)}-m_2(-,e_x)=m_1\cdot\mathcal{H}+\mathcal{H}\cdot m_1,
\end{align}
or there exists \emph{$\mathcal{H'}:\A(y,x)\to\A(y,x)[-1]$} such that 
\begin{align*}
\Id_{\A(y,x)}-m_2(e_x,-)=m_1\cdot\mathcal{H}'+\mathcal{H}'\cdot m_1.
\end{align*}
\end{defn}
The definition above is due to Lyubashenko (see \cite[Definition 7.3]{Lyu}).\\ 
We have the following inclusions:
\begin{align}\label{inclusone}
\mbox{strictly unital $\Ain$-cats}\subset\mbox{unital $\Ain$-cats}\subset \mbox{cohomological unital $\Ain$-cats}.
\end{align}
Of course fixed two strictly unital (resp. unital, cohomological unital) $\Ain$-categories we have the corresponding notion of strictly unital (resp. unital, cohomological unital) $\Ain$-functors. See \cite[pp 23]{Sei} and \cite[Definition 8.1 and Proposition 8.2]{Lyu}.\\
\\
From now we will use the following notation:\ 
\begin{itemize}
\item $\aCat^{\star}$ denotes the set of $\dagger$ $\Ain$-categories with $\star$ $\Ain$-functors.
\item $\aCat^{\star}_{\tiny\DgCat^{\star}}$ denotes $\dagger$ DG-categories with $\Ain$-functors.
\item $\DgCat^{\star}$ denotes the category of $\dagger$ DG-categories with strict $\dagger$ $\Ain$-functors.
\end{itemize}
Here $(\star,\dagger)\in\mathcal{f} (\emptyset, \mbox{strictly unital}), (u, \mbox{unital}), (c,\mbox{cohomological unital})\mathcal{g}$.\\ 
We have the following inclusions:
\begin{align}\label{inclustwo}
\DgCat^{\star}\subset \aCat^{\star}_{\tiny\DgCat^{\star}} \subset \aCat^{\star}.
\end{align}

\subsection{Equivalences of $\Ain$-categories}


In this subsection we give three notions of equivalence in the framework of $\Ain$-categories.\ 

First we need to recall the definition of prenatural transformation:\ 
we fix two $\Ain$-categories $\A$ and $\B$, and two $\Ain$-functors $\F,\G:\A\to\B$, a prenatural transformation 
$T=(T^0,T^1,...):\F\Rightarrow\G$ of degree $g$ is given by the datum:
\begin{itemize}
\item[1.] For every object $x\in\A$, a morphism $T^0(x):\F^0(x)\to\G^0(x)$ in $\B$.
\item[2.] for every $n>0$ and any sequence of objects $x_0,...,x_n\in\A$,
\begin{align*}
T^n:\A(x_{n-1},x_n)\otimes...\otimes\A(x_0,x_1)\to&\B\big(\F^0( x_0),\G^0( x_n)\big)[g-n].
\end{align*}
\end{itemize}
 
If $\A$ is a strictly unital $\Ain$-category then a prenatural transformation $T$ is \emph{strictly unital} if $T^n(f_n,...f_1)=0$ whenever $n\ge 1$ and there exists $j$ such that $f_j=e_x$ for any $x\in\A$.\\
\\
Given $\A$, $\B$ two unital (resp. non unital) $\Ain$-categories by \cite[\S5]{Lyu} we can equip the set of (resp. non unital) $\Ain$-functors from $\A$ to $\B$ with an $\Ain$-structure whose morphisms are the prenatural transformations.\ We denote by
\begin{align*}
\mbox{Fun}^{u}_{\infty}(\A,\B)\mbox{ $\big($resp. $\mbox{Fun}^{nu}_{\infty}(\A,\B)$$\big)$}
\end{align*}
such a unital (resp. non unital) $\Ain$-category, the corresponding $\Ain$-structure is denoted by $\mathfrak{M}^{n}$.\
The precise formulas of $\mathfrak{M}^1$ and $\mathfrak{M}^2$ can be found in \cite[Definition 1.3.2-1.3.3]{Orn1} or in \cite[(1d)]{Sei}.\\
Moreover, if $\A$ and $\B$ are strictly unital then 
\begin{align*}
\mbox{Fun}^{su}_{\infty}(\A,\B)
\end{align*}
denotes the strictly unital $\Ain$-category whose objects are the strictly unital $\Ain$-functors and whose morphisms are the strictly unital prenatural transformations.\\
\\
Given a $\star$ $\Ain$-functor $\F:\C\to\A$ we have two induced $\star$ $\Ain$-functors
\begin{align*}
\mathscr{R}_{\F}:\mbox{Fun}^{\dagger}_{\infty}(\A,\B)\to\mbox{Fun}^{\dagger}_{\infty}(\C,\B).
\end{align*}
and 
\begin{align*}
\mathscr{L}_{\F}:\mbox{Fun}^{\dagger}_{\infty}(\B,\C)\to\mbox{Fun}^{\dagger}_{\infty}(\B,\A).
\end{align*}
where $(\star,\dagger)\in\mathcal{f}\mbox{(strictly unital, $su$),(unital, $u$),(non unital, $nu$)}\mathcal{g}$, see \cite[(1e)]{Sei}.\\

Moreover we have an inclusion of $\Ain$-categories
\begin{align*}
\mbox{Fun}^{su}_{\infty}(\A,\B)\hookrightarrow \mbox{Fun}^u_{\infty}(\A,\B)=\mbox{Fun}^c_{\infty}(\A,\B).
\end{align*}
If $\A$ has split unit (see subsection \ref{nice}) the inclusion above is a quasi-equivalence according to the following definition.\\

We say that two $\Ain$-functors $\F,\G:\A\to\B$ (such that $\F^0=\G^0$) are \emph{homotopic}, in formula $\F\sim\G$, if there exists a prenatural transformation $T$ such that, $T^0=0$ and:
\begin{align*}
\F-\G=\mathfrak{M}^1(T).
\end{align*}
$\sim$ is an equivalence relation that is compatible with the left and right composition \cite[(1h)]{Sei}.\\

An $\Ain$-functor $\F:\A\to\B$ is a \emph{quasi-equivalence} if:
\begin{itemize}
\item[i)] $\F^1:\A(x,y)\to\B(\F^0(x),\F^0(y))$ is a quasi-isomorphism for any $x,y\in\A$.
\item[ii)] $H^0(\F):H^0(\A)\to H^0(\B)$ is essentially surjective.
\end{itemize}
Or equivalently the functor $\mbox{Ho}(\F)$ is an equivalence of graded categories.\\
If $\A$ and $\B$ are $\Ain$-algebras (i.e. categories with one object) then an $\Ain$-functor is a \emph{quasi-isomorphism} if condition i) holds.\\

Given another $\Ain$-functor $\G:\A\to\B$ we say that they are \emph{weakly equivalent} (denoted by $\F\approx\G$) if they are isomorphic (as objects) in {$H(\Fun(\A,\B))$}.\ 
Namely there exist two natural transformations $T:\F\Rightarrow\G$, $S:\G\Rightarrow\F$ and two prenatural trasformations $H:\F\Rightarrow\F$, $H':\G\Rightarrow\G$ such that: 
\begin{align}\label{appross}
\mathfrak{M}^2(S,T)=\Id_{\F}+\mathfrak{M}^1(H)\mbox{ and } \mathfrak{M}^2(T,S)=\Id_{\G}+\mathfrak{M}^1(H')
\end{align}

\begin{thm}[\cite{Orn2}]
Let $\F,\G:\A\to\B$ be two strictly unital $\Ain$-functors then
\begin{align*}
\F\sim\G\Rightarrow \F\approx\G \Rightarrow [\F]=[\G] \mbox{ in $\Ho(\aCat)$}.
\end{align*}
\end{thm}
\begin{namedthm}[\ref{patturzo}]
Let $\F,\G:\A\to\B$ be two unital $\Ain$-functors between unital $\Ain$-categories then
\begin{align*}
\F\approx\G \Rightarrow [\F]=[\G] \mbox{ in $\Ho(\aCat^u)$}.
\end{align*}
\end{namedthm}

\begin{defn}\label{equivalonze}
Let $\A$ and $\B$ be two $\Ain$-categories.\ They are said:
\begin{itemize}
\item[1.] \emph{Equivalent}, if there exist two $\Ain$-functors $\F:\A\to\B$ and $\G:\B\to\A$ such that the following diagrams
\begin{align*}
\xymatrix{
\A\ar[d]_{\F}\ar[dr]^{\tiny\mbox{Id}_{\A}}&\\
\B\ar[r]_{\G}&\A
}
\mbox{ , } 
\xymatrix{
\B\ar[dr]_{\tiny\mbox{Id}_{\B}}\ar[r]^{\G}&\A\ar[d]^{\F}\\
&\B
}
\end{align*}
are commutative.\ We denote two $\Ain$-categories equivalent by $\A\cong\B$.
\item[2.] \emph{Weakly-equivalent}, if there exist two $\Ain$-functors $\F:\A\to\B$ and $\G:\B\to\A$ such that $\G\cdot\F\approx\mbox{Id}$ and $\F\cdot\G\approx\mbox{Id}$.\ It means that the following diagrams
\begin{align*}
\xymatrix{
\A\ar[d]_{\F}\ar@/^1.2pc/[dr]^{\tiny\mbox{Id}_{\A}}&\ar@{}[dl]<1.2ex>^(.75){}="a"^(.35){}="b" \ar@{=>} "a";"b"\\
\B\ar@{}[ur]<1.2ex>^(.65){}="a"^(.25){}="b" \ar@{=>} "a";"b"\ar[r]_{\G}&\A
}
\mbox{ , } 
\xymatrix{
\B\ar@/_1.2pc/[dr]_{\tiny\mbox{Id}_{\B}}\ar[r]^{\G}&\A\ar@{}[ld]<1.2ex>^(.65){}="a"^(.25){}="b" \ar@{=>} "a";"b" \ar[d]^{\F}\\
\ar@{}[ur]<1.2ex>^(.75){}="a"^(.35){}="b" \ar@{=>} "a";"b" &\B 
}
\end{align*}
are commutative up to $\approx$.\ We denote two $\Ain$-categories equivalent by $\A\approx\B$.
\item[3.] \emph{Quasi-equivalent}, if there exist:
\begin{itemize}
\item[3i.] an integer $n\ge 0$.
\item[3ii.] A sequence of $\Ain$-categories $\A_i$, for $i=0,...,{n+1}$, such that $\A_0=\A$ and $\A_{n+1}=\B$.
\item[3iii.] A sequence of quasi-equivalences $\F_i$ such that $\F_i:\A_i\to \A_{i+1}$ or $\F_i:\A_{i+1}\to \A_{i}$, for $i=1,...,n$.
\end{itemize}
We denote two $\Ain$-categories quasi-equivalent by $\A\simeq\B$.
\end{itemize}
\end{defn}


\subsection{Equivalent homotopy categories}

Given $\DgCat^{\star}$, $\aCat^{\star}_{\tiny\DgCat^{\star}}$ and $\aCat^{\star}$ we define their homotopy category (denoted by $\Ho$) the corresponding localization with respect to the class of quasi-equivalences.\ We have the following result:

\begin{thm}[\cite{COS2}]\label{equivalonze}
The inclusions (\ref{inclusone}) and (\ref{inclustwo}) give rise to the following equivalences of categories:
\begin{align}\label{categoriez}
\xymatrix{
\Ho(\DgCat) \ar[d]^{\simeq}\ar[r]^-{\simeq} &\Ho(\aCat_{\tiny\DgCat}) \ar[d]^{\simeq}\ar[r]^-{\simeq} &\Ho(\aCat) \ar[d]^{\simeq}\\
\Ho(\DgCat^{u}) \ar[d]^{\simeq}\ar[r]^-{\simeq} &\Ho(\aCat^{u}_{\tiny\DgCat^{u}}) \ar[d]^{\simeq}\ar[r]^-{\simeq} &\Ho(\aCat^{u}) \ar[d]^{\simeq}\\
\Ho(\DgCat^c) \ar[r]^-{\simeq} &\Ho(\aCat^{c}_{\tiny\DgCat^{c}}) \ar[r]^-{\simeq} &\Ho(\aCat^{c}) 
}
\end{align}
\end{thm}

We recall that an h-projective $\Ain$-category (resp. DG-category) $\A$ is an $\Ain$-category (resp. DG-category) such that the hom-space $\A(x,y)$ is a h-projective DG module for every $x,y\in\A$.\
An $\emph{h-projective}$ resolution of $\A$ is the datum $(\A^{\tiny\mbox{hp}},\Psi_{\A})$, where $\A^{\tiny\mbox{hp}}$ is an h-projective $\Ain$category (resp. DG-category) and $\Psi_{\A}:\A^{\tiny\mbox{hp}}\to\A$ is a quasi-equivalence.\

If $\B$ is a unital/strictly unital DG-category/$\Ain$-category it is always possible to find an h-projective resolution.\
Actually we have more (see \cite[Theorem B and B$'$]{Orn2}):

\begin{thm}\label{risoluzioni}
Let $\A$ be a unital $\Ain$-category (resp. DG-category), there exists a (functorial up to $\approx$) h-projective unital $\Ain$-category (resp. DG-category) $\A^{\tiny\mbox{cm}}$ and a unital strict $\Ain$-functor $\Psi_{\A}:\A^{\tiny\mbox{cm}}\to\A$ surjective on the morphisms.\ Moreover, given a unital $\Ain$-functor $\F:\A\to\B$ we have a unital $\Ain$-functor $\F^{\tiny\mbox{cm}}$ fitting the diagram
\begin{align*}
\xymatrix{
\A^{\tiny\mbox{cm}}\ar[r]^{\F^{\tiny\mbox{cm}}}\ar[d]^{\Psi_{\A}}&\B^{\tiny\mbox{cm}}\ar[d]^{\Psi_{\B}}\\
\A\ar[r]^{\F}&\B
}
\end{align*}
which is commutative up to $\approx$, i.e $\Psi_{\B}\cdot\F^{\tiny\mbox{cm}} \approx \F\cdot\Psi_{\A}$.
\end{thm}

\begin{thm}
Let $\A$ be a strictly unital $\Ain$-category, there exists a (functorial up to $\approx$) h-projective strictly unital $\Ain$-category $\A^{\tiny\mbox{sf}}$ with split unit and a strict strictly unital $\Ain$-functor $\Psi_{\A}:\A^{\tiny\mbox{sf}}\to\A$ surjective on the morphisms.\ Moreover, given a strictly unital $\Ain$-functor $\F:\A\to\B$ we have a strictly unital $\Ain$-functor $\F^{\tiny\mbox{sf}}$ fitting the diagram
\begin{align*}
\xymatrix{
\A^{\tiny\mbox{sf}}\ar[r]^{\F^{\tiny\mbox{sf}}}\ar[d]^{\Psi_{\A}}&\B^{\tiny\mbox{sf}}\ar[d]^{\Psi_{\B}}\\
\A\ar[r]^{\F}&\B
}
\end{align*}
which is commutative up to $\approx$, i.e $\Psi_{\B}\cdot\F^{\tiny\mbox{sf}} \approx \F\cdot\Psi_{\A}$.
\end{thm}

\begin{notat}
We denote by: 
\begin{itemize}
\item[]
\begin{itemize}
\item[]
\begin{itemize}
\item[${\aCat_{\tiny\mbox{hps}}}$]: the category of the h-projective with split unit $\Ain$-categories, whose morphisms are the strictly unital $\Ain$-functors. 
\item[${\aCat_{\tiny\mbox{hps}}}/\approx$]: the category ${\aCat_{\tiny\mbox{hps}}}$, whose morphisms are given by the quotient ${\aCat_{\tiny\mbox{hps}}}/\approx(\A,\B):={\aCat_{\tiny\mbox{hps}}}(\A,\B)/\approx$. 
\item[${\aCat^u_{\tiny\mbox{hps}}}$]: the category of the h-projective with split unit $\Ain$-categories, whose morphisms are the unital $\Ain$-functors. 
\item[${\aCat^u_{\tiny\mbox{hps}}}/\approx$]: the category ${\aCat^u_{\tiny\mbox{hps}}}$, whose morphisms are given by the quotient ${\aCat^u_{\tiny\mbox{hps}}}/\approx(\A,\B):={\aCat^u_{\tiny\mbox{hps}}}(\A,\B)/\approx$. 
\item[${\aCat_{\tiny\mbox{hp}}}$]: the category of the h-projective $\Ain$-categories, whose morphisms are the strictly unital $\Ain$-functors.
\item[${\aCat_{\tiny\mbox{hp}}}/\approx$]: the category ${\aCat_{\tiny\mbox{hp}}}$, whose morphisms are given by the quotient ${\aCat_{\tiny\mbox{hp}}}/\approx(\A,\B):={\aCat_{\tiny\mbox{hp}}}(\A,\B)/\approx$. 
\item[${\aCat^u_{\tiny\mbox{hp}}}$]: the category of the h-projective $\Ain$-categories, whose morphisms are the unital $\Ain$-functors. 
\item[${\aCat^u_{\tiny\mbox{hp}}}/\approx$]: the category ${\aCat^u_{\tiny\mbox{hp}}}$, whose morphisms are given by the quotient ${\aCat^u_{\tiny\mbox{hp}}}/\approx(\A,\B):={\aCat^u_{\tiny\mbox{hp}}}(\A,\B)/\approx$. 
\end{itemize}
\end{itemize}
\end{itemize}
\end{notat}

\begin{thm}[\cite{COS2}]\label{fofofofofff}
The categories in (\ref{categoriez}) are equivalent to the categories $\aCat_{\tiny\mbox{hps}}/\approx$ and $\aCat^u_{\tiny\mbox{hps}}/\approx$  but they are not equivalent to $\aCat_{\tiny\mbox{hp}}/\approx$.
\end{thm}

On the other hand, we have the following result
\begin{thm}[\cite{Orn2}]
If $\A$ is h-projective unital $\Ain$-category we have an isomorphism of sets:
\begin{align*}
\Ho(\aCat)(\A,\B)\simeq\aCat^u(\A,\B)/_{\approx}.
\end{align*}
Moreover, if $\A$ is strictly unital with split unit $\Ain$-category (e.g. a semi-free $\Ain$-category) then:
\begin{align*}
\Ho(\aCat)(\A,\B)\simeq\aCat(\A,\B)/_{\approx}.
\end{align*}
\end{thm}

\newpage

\section{Homological Perturbation Theory and Main Construction}\label{TENSHPL}

In this section we give some important properties of the $\Ain$-categories.

\begin{defn}[Formal diffeomorphism]
Let $\mathsf{Q}$ be a graded quiver, a \emph{formal diffeomorphism} $\phi=\mathcal{f} \phi_1, \phi_2, ... \mathcal{g}$ is a family of multilinear maps 
\begin{align*}
\phi_{n}:\mathsf{Q}(x_{n-1},x_n)\otimes...\otimes\mathsf{Q}(x_0,x_1)&\to \mathsf{Q}(x_0,x_n)[1-n],
\end{align*}
for every sequence of objects $x_0,...x_n\in\mathsf{Q}$ and $n\ge1$, such that $\phi_1$ is an automorphism.
\end{defn}

We fix a graded quiver $\mathsf{Q}$, we have a map of sets:
\begin{align*}
\xymatrix{
\Biggl\{ \substack{\mbox{Formal Diffeomorphism} \\ \mbox{on $\mathsf{Q}$}}\Biggl\}\times \Biggl\{ \substack{\mbox{$\Ain$-structure} \\ \mbox{on $\mathsf{Q}$}}\Biggl\}\ar[r]^{}&\Biggl\{ \substack{\mbox{$\Ain$-structure} \\ \mbox{on $\mathsf{Q}$}}\Biggl\}
}
\end{align*}
Let $\A$ be an $\Ain$-structure on $\mathsf{Q}$ and $\phi$ a formal diffeomorphism on $\mathsf{Q}$ then we can find an $\Ain$-structure $\phi_{*}\A$ on $\mathsf{Q}$ making $\phi:\A\to\phi_{*}\A$ an $\Ain$-functor.\\
\\
The set of formal diffeomorphisms is a group taking as a composition the one in (\ref{compo}).

\begin{thm}\label{equivalone}
Let $\F:\A\to\B$ be an $\Ain$-functor then $\F$ is an equivalence of $\Ain$-categories if and only if $\F^1:|\A|\to|\B|$ is an equivalence of DG quivers.
\end{thm}
\begin{proof}
We need to prove that, if $\F^1$ is an equivalence of DG-quivers then $\F$ is an equivalence of $\Ain$-categories.\ 
To find the inverse it suffices to unwind formula (\ref{compo}). 
\end{proof}


\begin{thm}[]\label{HPT}
Let $\B$ be an $\Ain$-category, and $\Q$ a DG-quiver with the same objects of $\B$.\ Suppose that for every objects $x,y\in\B$ we have:
\begin{align*}
\xymatrix{
{\Q}(x,y)\ar@<0.5ex>[r]^-{f^1}&\ar@<0.5ex>[l]^-{g^1} {\B}(x,y) \ar@(r,d)[]^{T^1}
}
\end{align*}
where $f^1$, $g^1$ are chain maps and $T^1$ is a morphism of $R$-modules of degree $-1$ such that:
\begin{align*}
m_{\B}^1T^1 + T^1m_{\Q}^1=f^1\cdot g^1-\Id_{{\B}(x,y)}.
\end{align*}
There exists an $\Ain$-structure on $\Q$, two $\Ain$-functors $\F:=(f^0,f^1,\F^2,...)$, $\G:=(g^0,g^1,\G^2,...)$ and a prenatural transformation $T=(T^0,T^1,...)$ such that 
$$\F\cdot\G-\Id_{\B}=\mathfrak{M}^1(T).$$
\end{thm}

Theorem \ref{HPT} corresponds to \cite[(1.18)]{Sei}.\
Seidel provides an explicit description of the $\Ain$-functor $\F$ and the $\Ain$-structure $m_{\Q}$ given, for any $n\ge 2$, by the recursive formulas:
\begin{align*}
\F^n(f_n,...,f_1):=\displaystyle\sum_{r=2}^n\sum_{s_1,...,s_r=n} T^1\big( m^r_{\B}(\F^{s_r}(f_{n},...,f_{n-s_r+1}),...,\F^{s_1}(f_{s_1},...,f_1)),
\end{align*}
and
\begin{align*}
m_{\Q}^n(f_n,...,f_1):=\displaystyle\sum_{r=2}^n\sum_{s_1,...,s_r=n} \G^1\big( m^r_{\B}(\F^{s_r}(f_{n},...,f_{n-s_r+1}),...,\F^{s_1}(f_{s_1},...,f_1)).
\end{align*}
Where $f_j\in\Q(x_{j-1},x_j)$ for any $j=1,...,n$.\\


We also recall \cite[8.8 Theorem]{Lyu}: 

\begin{thm}\label{DC}
Let $\mathscr{F}$ be a unital $\Ain$-functor $\F:\A\to\B$ where $\A$ and $\B$ two unital $\Ain$-categories.\
We assume the following:
\begin{itemize}
\item[1.] {$\mathscr{F}^0:\mbox{Ob}(\A)\to\mbox{Ob}(\B)$} is a bijection of sets.
\item[2.] For every two objects $a$ and $a'$ in $\A$, the map of chain complexes:
{$$\mathscr{F}^1:\A(a,a')\to\B(\mathscr{F}^0(a),\mathscr{F}^0(a'))$$}
has a homotopy inverse $\mathscr{G}^1$ with homotopy $T^1$: 
{$${g}^1:\B(\mathscr{F}^0(a),\mathscr{F}^0(a'))\to\A(a,a').$$}
\end{itemize}
Setting ${g}^0=(\mathscr{F}^0)^{-1}$, we can extend ${g}^1$ to a unital $\Ain$-functor $\mathscr{G}:\B\to\A$ such that $\mathscr{F}$ and $\mathscr{G}$ are quasi-inverse to each other.\ Namely {$\mathscr{G}\cdot\mathscr{F}\approx\Id_{\mathscr{A}}$} and {$\mathscr{F}\cdot\mathscr{G}\approx\Id_{\mathscr{B}}$}.
\end{thm}

Assuming that $\B$ is a DG-category we can give an explicit description of the $\Ain$-functor $\G$ and the prenatural transformation $T$ of Theorem \ref{DC} as follows:
\begin{align*}
\G^n(f_1,...,f_n):=\G^1\big(\nu^n(f_n,...,f_1)\big),
\end{align*}
and 
\begin{align*}
T^n(f_1,...,f_n):=T^1\big(\nu^n(f_n,...,f_1)\big).
\end{align*}
Where 
\begin{align*}
\nu^n(f_n,...,f_1)&:=m_{\B}^2(T^{n-1}(f_n,...,f_2),f_1)\\
&+\displaystyle\sum^{n-1}_{j=1}m_{\B}^2((\F\cdot\G)^{n-j}(f_n,...,f_{n-j}),T^j(f_j,...,f_1))\\
&+\displaystyle\sum^{n-1}_{j=1}T^{n-1}(f_n,...,m_{\B}^2(f_{j+1},f_j),...,f_1).
\end{align*}
Where $f_j\in\Q(x_{j-1},x_j)$ for any $j=1,...,n$.




\subsection{Main Construction}

In this section we describe a general method to provide an $\Ain$-structure on the DG-quiver $(\A\otimes\B, m^1_{\A\otimes\B})$ 
where the differential $m^1_{\A\otimes\B}$ is the usual differential of DG complex.\
Now we can consider $\A$ and $\B$ as DG-quivers just forgetting the multilinear maps $m^{n>1}$.\\ 
This corresponds to $|\A|$ and $|\B|$ of \cite[(7)]{Orn1}.\ So we can form the DG-quiver $|\A|\otimes|\B|$:
\begin{itemize}
\item[1.] the objects are given by $\mbox{Obj}(\A)\times\mbox{Obj}(\B)$.
\item[2.] The morphisms are  
\begin{align*}
{\A\otimes\B}\big((a_1,b_1),(a_2,b_2)\big):={\A}(a_1,a_2)\oplus {\B}(b_1,b_2).
\end{align*}
\item[3.] The differential is given as follows:
\begin{align*}
m^1_{\A\otimes\B}(f,g):=m^1_{\A}(f)\otimes g +(-1)^{\tiny\mbox{deg}(f)} f\otimes m^1_{\B}(g).
\end{align*}
\end{itemize}

First we need a few of preliminaries.\ Lemma \ref{stupido} is a well known result, Lemmas \ref{stupidone} and \ref{stupidonzio} are new in the literature.\ 
They will be crucial to have a good definition of the tensor product of two $\Ain$-functors up to the homotopy relations 
$\sim$ and $\approx$.\

\begin{lem}\label{stupido}
We fix $X^{\bullet}$, $Y^{\bullet}$ and $Z^{\bullet}$ three chain complexes of $R$-modules.\\
Let $f,f':X^{\bullet}\to Y^{\bullet}$ and $g,g':Y^{\bullet}\to Z^{\bullet}$ be two morphisms of chain complexes.
\begin{itemize}
\item[1.] If $f\sim f'$ and $g\sim g'$ then $g\cdot f\sim g'\cdot f'$,
\item[2.] If $f\sim f'$ and $h,h':W^{\bullet}\to J^{\bullet}$ two morphisms of chain complexes such that $h\sim h'$
then $h\otimes f\sim h'\otimes f'$.
\end{itemize}
\end{lem}

\begin{proof}
Let $S$ be the homotopy of $f\sim f'$ and $T$ the homotopy of $g\sim g'$.
\begin{itemize}
\item[1.] As an intermediate step we prove $g\cdot f \sim g'\cdot f$.\\
To prove that we can choose as homotopy for example: 
\begin{align*}
H:=T\cdot f.  
\end{align*}
Then we can prove $g'\cdot f \sim g'\cdot f'$ taking the homotopy
\begin{align*}
H':=g'\cdot S.
\end{align*}
\item[2.] Let $R$ be the homotopy of $h\sim h'$.\ By the previous item it suffices to prove that $h\otimes \mbox{id}_{W^{\bullet}}\sim h'\otimes \mbox{id}_{W^{\bullet}}$ and $\mbox{id}_{J^{\bullet}}\otimes f\sim \mbox{id}_{J^{\bullet}}\otimes f'$.\ We can take the homotopy $R\otimes \mbox{id}_{W^{\bullet}}$ and $\mbox{id}_{J^{\bullet}}\otimes S$.
\end{itemize}
\end{proof}

\begin{lem}\label{stupidone}
Given three (non unital) DG-categories $\A$, $\B$ and $\C$ and two (non unital) DG-functors, 
$$\mathsf{F},\mathsf{G}:\A\to\B$$
if $\mathsf{F}\sim\mathsf{G}$ then $\mathsf{F}\otimes\Id_{\C}\sim\mathsf{G}\otimes\Id_{\C}$ and $\Id_{\C}\otimes\mathsf{F}\sim\Id_{\C}\otimes\mathsf{G}$.
\end{lem}

\begin{proof}
If $\mathsf{F}\sim\mathsf{G}$ there exists $S$ such that 
\begin{align*}
\mathsf{F}-\mathsf{G}=\mathfrak{M}^1(S).
\end{align*}
We can define a prenatural transformation 
$$\tilde{S}:\mathsf{F}\otimes\Id_{\C}\Rightarrow\mathsf{G}\otimes\Id_{\C}$$ 
as follows:
\begin{align*}
\tilde{S}^n(a_1\otimes c_1,...,a_n\otimes c_n):=(-1)^n\big( S^n(a_1,...,a_n)\otimes (c_1\cdot...\cdot c_n)\big).
\end{align*}
We calculate:
\begin{align*}
\big(\mathfrak{M}^1(\tilde{S})\big)^1(a_1\otimes c_1)&=d_{\B\otimes\C}\big(\tilde{S}^1(a_1\otimes c_1)\big)+ \tilde{S}^1\big(d_{\A\otimes\C}(a_1\otimes c_1)\big) \\
&=d_{\B\otimes\C}\big({S}^1(a_1)\otimes c_1\big)+ \tilde{S}^1\big(d_{\A}(a_1)\otimes c_1+(-1)^{\tiny\mbox{deg}(c_1)} a_1\otimes d_{\C}(c_1)\big) \\
&=d_{\B}({S}^1(a_1))\otimes c_1+(-1)^{\tiny\mbox{deg}({c}_1)} {S}^1(a_1)\otimes d_{\C}(c_1) +\\
&+ \tilde{S}^1\big(d_{\A}(a_1)\otimes c_1+(-1)^{\tiny\mbox{deg}(c_1)} a_1\otimes d_{\C}(c_1)\big) \\
&=d_{\B}({S}^1(a_1))\otimes c_1 + {S}^1 (d_{\A}(a_1))\otimes c_1 \\
&=\big(\mathfrak{M}^1(S)\big)^1(a_1)\otimes c_1\\
&=(\mathsf{F}-\mathsf{G})^1(a_1)\otimes c_1\\
&=(\mathsf{F}\otimes\Id_{\C}-\mathsf{G}\otimes\Id_{\C})^1(a_1\otimes c_1).
\end{align*}
If $n>1$,
\begin{align*}
\big(\mathfrak{M}^1(\tilde{S})\big)^n(a_n\otimes c_n,...,a_1\otimes c_1)&= d_{\B\otimes\C}\big(\tilde{S}^n(a_n\otimes c_n,...,a_1\otimes c_1)\big)\\
&\pm\displaystyle\sum_{j=1}^n \tilde{S}^n\big(a_n\otimes c_n,...,d_{\A\otimes\C}(a_j\otimes c_j),...,a_1\otimes c_1\big) \\
&\pm \big(\mathsf{F}\otimes\Id\big) (a_n\otimes c_n)\cdot \tilde{S}^{n-1}(a_{n-1}\otimes c_{n-1},...,a_1\otimes c_1) \big)\\
&\pm \tilde{S}^{n-1}(a_n\otimes c_n,...,a_2\otimes c_2)\cdot (\mathsf{G}\otimes\Id_{\C})(a_1\otimes c_1)\\
&\pm\displaystyle\sum_{j=1}^{n-1} \tilde{S}^{n-1}\big(a_n\otimes c_n,...,\big((a_{j+1}\otimes c_{j+1})\cdot(a_j\otimes c_j)\big),...,a_1\otimes c_1\big) \\
&
\end{align*}
This is equal to 
\begin{align*}
\big(\mathfrak{M}^1(\tilde{S})\big)^n&(a_n\otimes c_n,...,a_1\otimes c_1)= d_{\B\otimes\C}\big({S}^n(a_n,...,a_1)\otimes (c_n\cdot ... \cdot c_1)\big)\\
&\pm\displaystyle\sum_{j=1}^n \tilde{S}^n \big((a_n\otimes c_n,...,d_{\A}(a_j)\otimes c_j,...,a_1\otimes c_1) + (a_n\otimes c_n,...,a_j\otimes d_{\C}(c_j),...,a_1\otimes c_1) \big) \\
&\pm \big(\mathsf{F}\otimes\Id\big) (a_n\otimes c_n)\cdot \big( {S}^{n-1}(a_{n-1},...,a_1)\otimes ( c_{n-1}\cdot...\cdot c_1) \big)\\
&\pm \big({S}^{n-1}(a_n,...,a_1)\otimes (c_n\cdot ...\cdot c_2) \big)\cdot (\mathsf{G}\otimes\Id_{\C})(a_1\otimes c_1)\\
&\pm\displaystyle\sum_{j=1}^{n-1} \tilde{S}^{n-1}\big(a_n\otimes c_n,...,\big((a_{j+1}\otimes c_{j+1})\cdot(a_j\otimes c_j)\big),...,a_1\otimes c_1\big) \\
&
\end{align*}
and
\begin{align*}
\big(\mathfrak{M}^1(\tilde{S})\big)^n(a_n\otimes c_n,...,a_1\otimes c_1)&= d_{\B}\big({S}^n(a_n,...,a_1)\big)\otimes (c_n\cdot ... \cdot c_1)\\
&\pm \big(\displaystyle\sum_{j=1}^n {S}^n (a_n,...,d_{\A}(a_j),...,a_1) \big)\otimes (c_n\cdot ...\cdot c_1)  \\
&\pm \big(\mathsf{F}(a_n)\cdot {S}^{n-1}(a_{n-1},...,a_1) \big) \otimes ( c_{n}\cdot...\cdot c_1) \\
&\pm \big({S}^{n-1}(a_n,...,a_2)\cdot \mathsf{G}(a_1)\big)  \otimes (c_n\cdot ...\cdot c_1)\\
&\pm\displaystyle\sum_{j=1}^{n-1} {S}^{n-1}(a_n,...,(a_{j+1}\cdot a_j),...,a_1)\otimes (c_n\cdot...\cdot(c_{j+1}\cdot c_j)\cdot...\cdot c_1) \\
&=\big(\mathfrak{M}^1({S})\big)^n(a_n,...,a_1)\otimes (c_n\cdot...\cdot c_1)\\
&=(\mathsf{F}-\mathsf{G})^n(a_n,...,a_1)\otimes (c_n\cdot...\cdot c_1)\\
&=0.
\end{align*}
So $\mathsf{F}\otimes\Id_{\C}\sim \mathsf{G}\otimes\Id_{\C}$ (via $\tilde{S}$) and $\Id_{\C}\otimes\mathsf{F}\sim \Id_{\C}\otimes\mathsf{G}$.\\
\end{proof}

\begin{lem}\label{stupidonzio}
Given a non unital DG-category $\A$, a unital DG-categories $\B$, a DG-category $\C$ and two (non unital) DG-functors, 
$$\mathsf{F},\mathsf{G}:\A\to\B$$
if $\mathsf{F}\approx\mathsf{G}$ then $\mathsf{F}\otimes\Id_{\C}\approx\mathsf{G}\otimes\Id_{\C}$ and $\Id_{\C}\otimes\mathsf{F}\approx\Id_{\C}\otimes\mathsf{G}$.
\end{lem}

\begin{proof}
If $\mathsf{F}\approx\mathsf{G}$, we have two natural transformations $S:\mathsf{F}\Rightarrow\mathsf{G}$, 
$T:\mathsf{G}\Rightarrow\mathsf{F}$
and two prenatural transformations $H$ and $H'$ such that:
$$\mathfrak{M}^1(H)=\mathfrak{M}^2(T,S)-\Id$$
and
$$\mathfrak{M}^1(H')=\mathfrak{M}^2(S,T)-\Id$$
We define the natural transformation $\tilde{S}:\mathsf{F}\otimes\Id_{\C}\Rightarrow\mathsf{G}\otimes\Id_{\C}$ as
$$\tilde{S}^0(x\otimes y):=S^0(x)\otimes 1_y$$
where $x\in\mbox{Ob}(\A)$ and $y\in\mbox{Ob}(\C)$.\\
For $n>0$, we have 
$$\tilde{S}^n(a_n\otimes c_n,...,a_1\otimes c_1):=(-1)^n \big(S^n(a_n,...,a_1)\otimes (c_n\cdot...\cdot c_1)\big)$$
we can define, in the same way the natural transformation $\tilde{T}:\mathsf{F}\otimes\Id_{\C}\Rightarrow\mathsf{G}\otimes\Id_{\C}$ 
and the prenatural transformations $\tilde{H}$, and $\tilde{H}'$.\\
Note that 
\begin{align*}
\big(\mathfrak{M}^2(\tilde{S},\tilde{T})\big)^n(a_n\otimes c_n,...,a_1\otimes c_1)&=\displaystyle\sum_{j=1}^{n-1}\tilde{S}^{n-j}(a_n\otimes c_n,...,a_{j+1}\otimes c_{j+1})\cdot\tilde{T}^j(a_j\otimes c_j,...,a_1\otimes c_1) \\
&+\tilde{S}^0(x_n\otimes y_n)\cdot\tilde{T}^n(a_n\otimes c_n,...,a_1\otimes c_1)\\
&+\tilde{S}^n(a_n\otimes c_n,...,a_1\otimes c_1)\cdot\tilde{T}^n(x_0\otimes y_0)\\
&=\big(\mathfrak{M}^2({S},{T})\big)^n(a_n,...,a_1)\otimes (c_n\cdot...\cdot c_1).
\end{align*}
By the previous calculation we have 
$$\mathfrak{M}^1(\tilde{H})=\mathfrak{M}^2(\tilde{T},\tilde{S})-\Id$$
and
$$\mathfrak{M}^1(\tilde{H'})=\mathfrak{M}^2(\tilde{S},\tilde{T})-\Id$$
and we are done.
\end{proof}

Combining Lemma \ref{stupido} and Theorem \ref{DC} we have:

\begin{thm}\label{enzone}
Let $\A_1$ and $\A_2$ two non unital $\Ain$-categories such that there exist: 
\begin{itemize}
\item[T1.] two non unital DG-categories $\C_{\A_1}$ and $\C_{\A_2}$ with the same objects respectively of $\A_1$ and $\A_2$,
\item[T2.] two chain maps 
\begin{align*}
f^1_{\A_i,x_i,y_i}:\A_i(x_i,y_i)\to\C_{\A_i}(x_i,y_i)
\end{align*}
and 
\begin{align*}
g^1_{\C_{\A_i},x_i,y_i}:\C_{\A_i}(x_i,y_i)\to {\A_i}(x_i,y_i)
\end{align*}
for every pair of objects $x_i,y_i\in\A_1$,
\item[T3.] A morphism of $R$-modules $T^1_i$ such that 
\begin{align}
m_{\C_{\A_i}}^1T^1 + T^1m_{\A_i}^1=f^1\cdot g^1-\Id_{{\C_{\A_i}}(x,y)}.
\end{align}
and 
\begin{align*}
g^1\cdot f^1=\Id_{\A_i(x,y)}.
\end{align*}
\end{itemize}
Then there exists an $\Ain$-structure on $\A_1\otimes \A_2$, which does not depends on the homotopy $T^1$, such that $\A_1\otimes\A_2\approx\C_{\A_1}\otimes\C_{\A_2}$ and two $\Ain$-functors 
\begin{align*}
\eta_{\A_1\otimes\A_2}:\A_1\otimes\A_2\to\C_{\A_1}\otimes\C_{\A_2}
\end{align*}
and 
\begin{align*}
\Xi_{\A_1\otimes\A_2}:\C_{\A_1}\otimes\C_{\A_2}\to\A_1\otimes\A_2
\end{align*}
such that $\eta_{\A_1\otimes\A_2}\cdot\Xi_{{\A_1\otimes\A_2}}\sim\Id_{\C_{\A_1}\otimes\C_{\A_2}}$.
\end{thm}

\begin{proof}
By Lemma \ref{stupido} we have an homotopy $H^1$.\ We have:
\begin{align*}
\xymatrix{
\big({|\A_1|\otimes|\A_2|}\big)(x,y)\ar@<0.5ex>[rr]^-{f^1_{\A_1}\otimes f^1_{\A_2}}&&\ar@<0.5ex>[ll]^-{g_{\C_{\A_1}}^1\otimes g_{\C_{\A_2}}^1} \big({\C_{\A_1}\otimes\C_{\A_2}}\big)(x,y) \ar@(r,u)[]_{H^1}
}
\end{align*}
By Theorem \ref{HPT} we find an $\Ain$-structure on $\A_1\otimes\A_2$ and two $\Ain$-functors
$$\eta_{\A_1\otimes\A_2}:\A_1\otimes\A_2\to\C_{\A_1}\otimes\C_{\A_2}$$ 
which extends $f^1_{\A_1}\otimes f^1_{\A_2}$ and  
$$\Xi_{\A_1\otimes\A_2}:\C_{\A_1}\otimes\C_{\A_2}\to\A_1\otimes\A_2$$ 
which extends $g^1_{\A_1}\otimes g^1_{\A_2}$ such that 
$$\eta_{\A_1\otimes\A_2}\cdot \Xi_{\A_1\otimes\A_2}\sim\Id_{\C_{\A_1}\otimes\C_{\A_2}}.$$
On the other hand by Theorem \ref{DC} we find 
$$g_{\C_{\A_1}}\otimes g_{\C_{\A_2}}: \C_{\A_1}\otimes\C_{\A_2}\to\A_1\otimes\A_2$$
which extends the functor of DG quivers $g_{\C_{\A_1}}^1\otimes g_{\C_{\A_2}}^1$ such that $$(g_{\C_{\A_1}}\otimes g_{\C_{\A_2}})\cdot(\eta_{\A_1\otimes\A_2})\approx\Id_{\A_1\otimes\A_2}$$
and 
$$\eta_{\A_1\otimes\A_2}\cdot(g_{\C_{\A_1}}\otimes g_{\C_{\A_2}})\approx\Id_{\C_{\A_1}\otimes\C_{\A_2}}.$$
To conclude we point out that the $\Ain$-structure on $|\A_1|\otimes|\A_2|$ does not depends on the homotopy $H^1$.\ 
If $\tilde{H}^1$ is another homotopy then the corresponding $\Ain$structure $\A_1\tilde{\otimes}\A_2$ is equivalent (as $\Ain$-category) to $\A_1{\otimes}\A_2$.\\
We denote by $f_{\A_1}\tilde{\otimes} f_{\A_2}:\A_1\tilde{\otimes}\A_2\to\C_{\A_1}\otimes\C_{\A_2}$ and by $g_{\C_{\A_1}}\tilde{\otimes}  g_{\C_{\A_2}}: \C_{\A_1}\otimes\C_{\A_2}\to\A_1\tilde{\otimes}\A_2$ the $\Ain$functors extending $f^1_{\A_1}\otimes f^1_{\A_2}$ and $g_{\C_{\A_1}}\otimes g_{\C_{\A_2}}$ via the homotopy $\tilde{H}^1$.\\
We have the following diagram:
\begin{align*}
\xymatrix{
\A_1\otimes\A_2 \ar@<0.5ex>[rr]^-{f_{\A_1}\otimes f_{\A_2}}&&\ar@<0.5ex>[ll]^-{g_{\C_{\A_1}}\otimes g_{\C_{\A_2}}} \C_{\A_1}\otimes\C_{\A_2}\ar@<0.5ex>[rr]^-{g_{\C_{\A_1}}\tilde{\otimes} g_{\C_{\A_2}}}&&\ar@<0.5ex>[ll]^-{f_{\A_1}\tilde{\otimes} f_{\A_2}}\A_1\tilde{\otimes}\A_2.
}
\end{align*}
Since $g^1\cdot f^1=\Id$ then by Theorem \ref{equivalone} $\A_1\otimes\A_2$ and $\A_1\tilde{\otimes}\A_2$ are equivalent as $\Ain$-categories.
\end{proof}





 

Moreover for simple reason of degrees we have:

\begin{cor}
If $\A$ and $\B$ are DG-categories then (every) $\Ain$-structure on $|\A|\otimes |\B|$ induced by Theorem \ref{enzone} is the usual tensor product defined in (\ref{Tensorello}).
\end{cor}

In the next subsection we give a natural way to associate a (non-unital) DG-category to an $\Ain$-one satisfying the conditions T1, T2 and T3.

\newpage

\section{Tensor product of two $\Ain$-categories and $\Ain$-functors}

In this section we define the tensor product of two $\Ain$-categories (Definition \ref{tenso}) making use of Theorem \ref{enzone}.\ 
We give two definitions of tensor product of $\Ain$-functors (Definition \ref{compofuncto}) and equation (\ref{sanzo}).\ 
This tensor product provides a symmetric monoidal structure (up to homotopy) on $\aCat$ (see subsection \ref{sms}) 
and a closed symmetric monoidal category on $\aCat^u/\approx$.

\subsection{Tensor product of two $\Ain$-categories}
We recall a few constructions in \cite{Orn1}, \cite{COS} and \cite{COS2}.\\
Let $\A$ be a non unital $\Ain$-category.\ 
The non unital DG-category $\Un(\A)$ is given by
$$\Un(\A):=\Omega\cdot \mbox{B}_{\infty}(\A)$$
where $\Omega$ is the Cobar construction (see \cite[\S 2.1]{Orn1}) and 
$\mbox{B}_{\infty}$ is the Bar construction (see \cite[\S 2.6]{Orn1}).\\
Explicitly the differential of $\Un(\A)$ is given by:
\begin{align*}
d_{\Un(\A)}\big(g_m[-1]\otimes...\otimes g_1[-1]\big)&=d_{\Un(|\A|)}\big(g_m[-1]\otimes...\otimes g_1[-1]\big)+\\
&+\displaystyle\sum_{i=1}^m (-1)^{\tiny\mbox{deg}(g_1)+...+\mbox{deg}(g_m)}g_m[-1]\otimes...\otimes g_{i+1}[-1]\otimes \\
&\otimes\displaystyle\sum_{k=1}^{n_i}\sum_{j=2}^{n_i-k+1}(-1)^{\tiny\mbox{deg}(f^i_1)+...+\mbox{deg}(f^i_{k-1})} \big( f^i_{n_i}[1]\otimes...\otimes f^i_{k+j}[1]\otimes \\
&\otimes m_j(f^i_{k+j-1},...,f^i_k)[1]\otimes f^i_{k-1}[1]\otimes ...\otimes f_1^i[1])[-1]\otimes g_{i-1}[-1]\otimes...\otimes g_1[-1]
\end{align*}
where
$$g_m[-1]\otimes...\otimes g_1[-1]:=(f^m_{n_m}[1]\otimes ... \otimes f^m_1[1])[-1]\otimes...\otimes (f^1_{n_1}[1]\otimes ... \otimes f^1_1[1])[-1].$$
We have an $\Ain$-functor:
\begin{align*}
{\alpha_{\A}}:\A&\to \Un(\A)
\end{align*}
Defined as:
\begin{align}\label{alfone}
\alpha^{n}_{\A}:(f_1,...,f_n)&\mapsto(f_1[1]\otimes...\otimes f_n[1])[-1]
\end{align}
On the other hand, if $\A$ is a non unital DG-category we have a non unital DG-functor 
\begin{align}\label{cocuz}
\xi_{\A}:\Un\big( i (\A) \big) \to\A.
\end{align}
Defined as follows:
$$\xi_{\A}((f^m_{n_m} [1]\otimes...\otimes  f^m_1[1])[-1]\otimes...\otimes (f^1_{n_1} [1]\otimes...\otimes  f^1_1[1])[-1])=
\begin{cases}
f^m_1\cdot...\cdot f^1_1, & \text{if $n_m=...=n_1=1$}\\
0, & \text{otherwise}.
\end{cases}
$$
We have an adjunction of categories:
\begin{align}\label{gorgogo}
\Un:\aCat^{\tiny\mbox{nu}} \rightleftharpoons \DgCat^{\tiny\mbox{nu}}:i
\end{align}
$\alpha$ and $\xi$ are respectively the unit and the counit of the adjunction $\Un\dashv i$.\\
\\
If $\A$ is an $\Ain$-(not DG-)category, (\ref{cocuz}) does not make sense, since the product (i.e. $m^2_{\A}$) is not associative.\ 
In this section we prove the following result:

\begin{thm}\label{adjun}
Fixed a non unital $\Ain$-category $\A$, we have:
\begin{itemize}
\item[T1.] the non-unital DG-category $\Un(\A)$ has the same objects of $\A$
\item[T2.-T3.] $\alpha^1_{\A}$ defined as:
\begin{align*}
\alpha_{\A}^1:\A(x_0,x_1)&\to i\cdot\Un(\A)(x_0,x_n),\\
f&\mapsto (f[1])[-1]
\end{align*}
has an homotopy inverse $\chi^1_\A$ (defined as \ref{CHI}).\ In particular:
\begin{align*}
\alpha^1_{\A}\cdot\chi^1_{\A}-\mbox{Id}_{\Un(\A)}=d_{\Un(\A)}\cdot H^1_{\A}+H^1_{\A}\cdot d_{\Un(\A)}.
\end{align*}
where $H^1_{\A}$ is defined in (\ref{H}) and
\begin{align*}
\chi^1_{\A}\cdot \alpha^1_{\A} =\mbox{Id}_{\A}.
\end{align*}
\end{itemize}
\end{thm}
Since Theorem \ref{adjun} fulfills the hypothesis of Theorem \ref{enzone} we have the following:
\begin{defn}\label{tenso}
Given two non unital $\Ain$-categories $\A$ and $\B$, we define $\A\otimes\B$ to be the graded quiver $|\A|\otimes|\B|$ with the $\Ain$-structure given by Theorem \ref{enzone} setting $f^1_{\A_i}=\alpha^1_{\A_i}$, $g^1_{\A_i}=\chi^1_{\A_i}$ and $H^1_{\A_i}$ (of Theorem \ref{adjun}). 
\end{defn}
Clearly by Theorem \ref{DC} we have that $\chi^1_\A$ extends to an $\Ain$-functor $\chi_{\A}:\Un(\A)\to\A$.\ 
We provide an explicit description of such a structure in Section \ref{explicit}.\
We note that, fixed $\alpha^1_{\A_i}$ and $\chi^1_{\A_i}$, then the $\Ain$-structure (obtained by Theorem \ref{enzone}) on $\A_1\otimes\A_2$ is defined up to equivalence.
\begin{proof}[Proof of Theorem \ref{adjun}]
We provide an explicit formula for the quasi-inverse $\chi^1_{\A}$ and the prenatural transformation $H^1_{\A}$.\
We fix two objects $x,y\in\A$, we have a filtration of $\Un(\A)(x,y)$ given by:
\[
\xymatrix{
F_{x,y}:=L_1(x,y)\subset ...\subset L_{n}(x,y)\subset ...\subset \Un(\A)(x,y).
}
\]
Here $L_n(x,y)$ is the DG $R$-module generated by the elements of the form 
\begin{align*}
(f^{j}_{m_j}[1]\otimes...\otimes f^{j}_1[1])[-1]\otimes...\otimes (f^{1}_{m_1}[1]\otimes...\otimes f^{1}_1[1])[-1]
\end{align*}
where $m_j+...+m_1\le n$.
We have a morphism of graded $R$-complexes:
$$R_{x,y}:U^{n}(\A)(x,y)\to U^{n}(\A)(x,y)[-1],$$
defined as follows:
\begin{align}\label{gurzo}
R_{x,y}\big(g_m[-1]\otimes...\otimes g_1[-1] \big):=
\begin{cases}
(f^{m}_{1}[1]\otimes f^{m-1}_{n_{m-1}}[1]\otimes...\otimes f^{m-1}_1[1])[-1]\otimes...\otimes g_1[-1], \mbox{ if $n_m=1$}\\
0, \mbox{ otherwise.}
\end{cases}
\end{align}
If $f\in L_{n}(x,y)$ then 
\begin{align}\label{lollz}
R_{x,y}\cdot d_{\Un(|\A|)}(f)+d_{\Un(|\A|)}\cdot R_{x,y}(f)=f.
\end{align}
Here $|\A|$ denotes the DG-quiver $(\A,m^1_{\A},0,0...)$ (the proof is in Appendix \ref{appendix}).\\
We have a morphism $r:L_n(x,y)\to L_{n-1}(x,y)$ of degree zero of chain complexes defined as
\begin{align*}
r(f):=f-d_{U^{n}(\A)}\cdot R_{x,y}(f)-R_{x,y}\cdot d_{U^{n}(\A)}(f).
\end{align*}
This is a morphism of chain complexes since
\begin{align*}
d_{\Un(\A)}r&:=d_{\Un(\A)}\big(\mbox{Id}_{\Un(\A)}(x,y)-d_{\Un(\A)}\cdot R_{x,y}-R_{x,y}\cdot d_{\Un(\A)} \big)\\
&=d_{\Un(\A)}-d_{\Un(\A)}\cdot R_{x,y}\cdot d_{\Un(\A)} \big)\\
&=\big(\mbox{Id}_{\Un(\A)(x,y)}-d_{\Un(\A)}\cdot R_{x,y}-R_{x,y}\cdot d_{\Un(\A)} \big)d_{\Un(\A)}\\
&=r d_{\Un(\A)}.
\end{align*}
If $f\in \Un(\A)(x,y)$ is homogeneous of length $n$, we define $H^1_{\A}$:
\begin{align}\label{H}
H^1_{\A}(f):=R_{x,y}(f)-R_{x,y}(r(f))+R_{x,y}(r^2(f))+...\pm R_{x,y}(r^{n-2}(f))=\displaystyle\sum_{j=0}^{n-2}(-1)^{j}R_{x,y}(r^j(f)).
\end{align}
This is a morphism of graded complexes $H^1_{\A}:U^{n}(\A)(x,y)\to U^{n}(\A)(x,y)$ of degree $-1$.\\
If $f$ is homogeneous of length $n$ then:
\begin{align*}
f-r^{n-1}(f)=d_{U^{n}(\A)}\cdot H^1_{\A}(f)+H^1_{\A}\cdot d_{U^{n}(\A)}(f).
\end{align*}
First we note $H^1_{\A}\cdot d_{U^{n}(\A)}(f)=\displaystyle\sum_{j=0}^{n-2}(-1)^{j}R_{x,y}\big(d_{U^{n}(\A)}(r^j(f)) \big)$ it follows:
\begin{align*}
d_{U^{n}(\A)}\cdot H^1_{\A}(f)+H^1_{\A}\cdot d_{U^{n}(\A)}(f)&= d_{U^{n}(\A)}\big(\displaystyle\sum_{j=0}^{n-2}(-1)^{j}R_{x,y}(r^j(f))\big) +  \displaystyle\sum_{j=0}^{n-2}(-1)^{j}R_{x,y}\big(d_{U^{n}(\A)}(r^j(f)) \big)\\
&=(f- r(f))+(r(f)-r^2(f))+...-r^{n-1}(f)\\
&=f-r^{n-1}(f). 
\end{align*}
We define $\chi^1_{\A}$ as follows:
\begin{align}\label{CHI}
\chi^1_{\A}(f):=\big(r^{n-1}(f)[-1]\big)[1]
\end{align}
Where $n$ is the length of $f$.\ Since $r$ is a chain complexes morphism of degree zero so is $\chi^1_{\A}$.
\end{proof}

\begin{rem}\label{chitild}
Let $\A$ be an $\Ain$-category then there exists an $\Ain$-functor $\tilde{\tilde{\chi}}_{\A}:\Un(\A)\to\A$ 
such that:
$$\alpha_{\A}\cdot\tilde{\tilde{\chi}}_{\A}\sim\Id_{\Un(\A)}$$
and 
$$\tilde{\tilde{\chi}}_{\A}\cdot \alpha_{\A}=\Id_{\A}.$$
First we consider the DG-category
$(\Un(\A),d_{\Un(\A)},\otimes)$.\\
We can take 
\begin{align*}
\xymatrix{
\big({\A}(x,y),m^1_{\A}\big)\ar@<0.5ex>[rr]^-{\alpha^1_{\A}}&&\ar@<0.5ex>[ll]^-{\chi^1_{\A}} \big({\Un(\A)}(x,y),d_{\Un(\A)}\big) \ar@(r,d)[]^{H^1}
}
\end{align*}
By Lemma \ref{HPT} we have an $\Ain$-structure $\tilde{\A}:=(\A,m^1_\A,\tilde{m}^2_\A,...,\tilde{m}^n_\A)$ and two $\Ain$-functors 
\begin{align*}
\tilde{\alpha}_\A:\tilde{\A}\to\Un(\A)
\end{align*}
and 
\begin{align*}
\tilde{\chi}_\A:\Un(\A)\to\tilde{\A}
\end{align*}
such that $\tilde{\chi}^1_{\A}=\chi^1_{\A}$, $\tilde{\alpha}^1_{\A}=\alpha^1_{\A}$ and $\tilde{\alpha}_\A\cdot\tilde{\chi}_\A\sim \Id_{\Un(\A)}$.\\
Since $\chi^1_\A\cdot\alpha^1_\A=\Id_\A$ by Lemma \ref{equivalone}, we have that $\tilde{\chi}_\A\cdot\alpha_\A:\A\to\tilde{\A}$ 
and ${\chi}_\A\cdot\tilde{\alpha}_\A:\tilde{\A}\to{\A}$ are equivalence of $\Ain$-categories.\\
We can define $\tilde{\tilde{\chi}}_{\A}:=(\tilde{\chi}_\A\cdot\alpha_{\A})^{-1}\cdot \tilde{\chi}_\A$.\\
Since 
$$\tilde{\alpha}_\A\cdot\tilde{\chi}_\A\sim\Id_{\Un(\A)}$$
we have: 
$$\tilde{\alpha}_\A\cdot\tilde{\chi}_\A\cdot\alpha_{\A}\sim\alpha_{\A}$$
so
$$\tilde{\alpha}_\A\sim\alpha_{\A}\cdot (\tilde{\chi}_\A\cdot\alpha_{\A})^{-1}.$$
Multiplying on the right by $\tilde{\chi}_\A$, we have:
$$\Id_{\Un(\A)}\sim\alpha_{\A}\cdot (\tilde{\chi}_\A\cdot\alpha_{\A})^{-1}\cdot \tilde{\chi}_\A$$
and we are done.
\end{rem}

\subsection{Properties of tensor product \ref{tenso}}

Clearly if $\A$ is cohomological unital (resp. cohomological unital) then also $\Un(\A)$ is so (since they are quasi-equivalent).\ 
In particular $\Un$ induces the adjunctions 
\begin{align*}
\xymatrix{
\aCat^{\tiny\mbox{u}} \ar@<0.5ex>[d]^-{\Un} \ar@^{^(->}[r]& \aCat^{\tiny\mbox{cu}} \ar@<0.5ex>[d]^-{\Un}\ar@^{^(->}[r]& \aCat^{\tiny\mbox{nu}}  \ar@<0.5ex>[d]^-{\Un}\\
\DgCat^{\tiny\mbox{u}} \ar@<0.5ex>[u]^-{i}  \ar@^{^(->}[r]&\DgCat^{\tiny\mbox{cu}}\ar@<0.5ex>[u]^-{i}  \ar@^{^(->}[r]& \DgCat^{\tiny\mbox{nu}}\ar@<0.5ex>[u]^-{i}
}
\end{align*}
If $\A$ and $\B$ are two unital (resp. cohomological unital) $\Ain$-categories then $\A\otimes\B$ is a unital (resp. cohomological unital) $\Ain$-category.\ 
Unfortunately this is no longer true if $\A$ is strictly unital.\ 
If $\A$ and $\B$ are strictly unital then $\A\otimes\B$ is not, a priori, strictly unital, since $\Un(\A)$ is not strictly unital.\ 
What we can do is to deform the $\Ain$-structure on $\A\otimes\B$ via formal diffeomorphisms in order to get an $\Ain$-category strictly unital which is isomorphic to the one provided by Definition \ref{tenso}.\ 

On the other hand, if $\A$ is a strictly unital $\Ain$-category, then $\Un(\A)$ is cohomological (non strictly) unital.\ 
Given an object $x\in\A$, the strictly unit $e_x$ is mapped to $(e_x[1])[-1]$ (by $\alpha_{\A}$) which is a cohomological (not strictly) unit of $x\in\Un(\A)$.\ 
In the same vein, if $\A$ and $\B$ are strictly unital $\Ain$-categories, then $\Un(\A)\otimes\Un(\B)$ is a cohomological (not strictly) unital DG-category.\ 
It implies that, a priori, $\A\otimes\B$ is not strictly unital.\ 

\begin{lem}
Let $(a,b)\in\A\otimes\B$ then the closed degree zero morphism 
\begin{align*}
\underline{e}:=e_{a}\otimes e_{b}
\end{align*}
where $e_{a}$ and $e_{b}$ are the identities of the objects $a$ and $b$ respectively, is such that 
\begin{align*}
m^2_{\A\otimes\B}(f\otimes g, \underline{e})=& f\otimes g
\end{align*}
and 
\begin{align*}
m^2_{\A\otimes\B}(\underline{e},f\otimes g)=& f\otimes g.
\end{align*}
For any $f\otimes g\in\A\otimes\B$.
\end{lem}

\begin{proof}
This follows directly from the calculation in Example \ref{mdue}.
\end{proof}

\begin{thm}
If $\A$ and $\B$ are strictly unital $\Ain$-categories then there exists an $\Ain$-category $\A{\otimes}\B$ 
which is isomorphic and has the same objects and morphisms of the $\Ain$-structure given by Definition \ref{tenso}. 
\end{thm} 

\begin{proof}
We can make $\underline{e}:=e_{a}\otimes e_{b}$ a strict unit for $\A\otimes\B$ via a formal diffeomorphism, see \cite[Lemma 2.1]{Sei}.
\end{proof}

\begin{rem}
Note that \cite[Lemma 2.1]{Sei} says that, every cohomological unital $\Ain$-category is isomorphic to a strictly unital $\Ain$-category.\ This is true only if $R$ is a field.\ But it is always true that if we have a cohomological unital $\Ain$-category $\A$ with a unit $e$ satisfying condition i) of Definition \ref{strictlyunital} then we can make $e$ satisfying condition ii).\ In other words there exists an $\Ain$-category $\tilde{\A}$ which is isomorphic to $\A$.
\end{rem}

We recall that "$\cong$" denotes an equivalence of DG or $\Ain$-categories.\ 
It is not difficult to prove the following:
\begin{thm}\label{grgrgr}
Let $\A_1$ be an $\Ain$category, we have 
\begin{itemize}
\item[1.] (\emph{Reflexive}) 
Denoting by $R$ the $\Ain$category with one object and one morphism, we have an equivalence of $\Ain$-categories:
\begin{align}\label{equi1}
I:\A_1\otimes R\cong \A_1.
\end{align}
\item[2.] (\emph{Symmetric}) Given another $\Ain$category $\A_2$, we have an equivalence of $\Ain$-categories:
\begin{align}\label{equi2}
\mu:\A_1\otimes \A_2\cong \A_2\otimes\A_1.
\end{align}
\item[3.] (\emph{Associative}) Let $\A_3$ another $\Ain$category, we have an equivalence of $\Ain$-categories:
\begin{align}\label{equi3}
\sigma:(\A_1\otimes\A_2)\otimes\A_3\cong \A_1\otimes(\A_2\otimes\A_3).
\end{align}
\end{itemize}
\end{thm}
The first equivalence is straightforward, the second and the third ones are directly recovered from Theorem \ref{equivalone}.\ 
Note that Theorem \ref{grgrgr} does not guaranteed the functoriality of such equivalences.\\

We have a symmetric monoidal structure on the category of DG-categories (with any kind of unit) the monoidal unit 
is the commutative base ring $R$:
\begin{itemize}
\item[1.] (\emph{Right-left unitor}) 
Let $\mathcal{C}$ be a DG-category, we have an equivalence of DG-categories:
\begin{align*}
_rI^{\tiny\mbox{DG}}_{\mathcal{C}_1}:\mathcal{C}_1\otimes R&\cong \mathcal{C}_1\\
(c_1\otimes r_1\cdot 1_R)&\mapsto r_1\cdot c_1.
\end{align*}
and $_lI^{\tiny\mbox{DG}}_{\mathcal{C}_1}:R\otimes \mathcal{C}_1\cong \mathcal{C}_1$.
\item[2.] (\emph{Swap map}) Given another DG-category $\mathcal{C}_2$, we have an equivalence of DG-categories:
\begin{align*}
\mu^{\tiny\mbox{DG}}_{\mathcal{C}_1 \mathcal{C}_2}:\mathcal{C}_1\otimes \mathcal{C}_2&\cong \mathcal{C}_2\otimes\mathcal{C}_1\\
c_1\otimes c_2&\mapsto  c_2\otimes c_1.
\end{align*}
\item[3.] (\emph{Associator}) Let $\mathcal{C}_3$ be another DG-category, we have an equivalence of DG-categories:
\begin{align*}
\sigma^{\tiny\mbox{DG}}_{\mathcal{C}_1 \mathcal{C}_2 \mathcal{C}_3}:\mathcal{C}_1\otimes(\mathcal{C}_2 \otimes\mathcal{C}_3)&\cong (\mathcal{C}_1\otimes \mathcal{C}_2)\otimes\mathcal{C}_3\\
c_1\otimes (c_2\otimes c_3)&\mapsto (c_1\otimes c_2)\otimes c_3.\\
\end{align*}
\end{itemize}
The DG-functors above satisfy the coherence conditions, the unit coherence, the associative coherence and the inverse law.

\subsection{Tensor product of two $\Ain$-functors}\label{functoroni}

In this section we will prove the properties of the tensor product found in the previous section 
and we define the tensor product of two non unital (resp. cohomological unital, unital) $\Ain$-functors.\\

First, we recall that, providing $\A\otimes\B$ with the $\Ain$-structure of Theorem \ref{HPT}, we have that
$\alpha^1_{\A}\otimes\alpha^1_{\B}$ extends to an $\Ain$-functor 
\begin{align*}
\eta_{\A\otimes\B}:\A\otimes\B\to \Un(\A)\otimes \Un(\B).
\end{align*}
In the same vein, we have an $\Ain$-functor 
\begin{align*}
\Xi_{\A\otimes\B}:\Un(\A)\otimes \Un(\B)\to \A\otimes\B.
\end{align*}
which extends $\chi^1_{\A}\otimes\chi^1_{\A}$ and a prenatural transformation $T$ which extends the homotopy $T^1$, see equation (\ref{torcofo}).\\
We have the diagram:
\begin{align*}
\xymatrix{
\A\otimes\B \ar@<0.5ex>[rr]^-{\eta_{\A\otimes \B}}&&\ar@<0.5ex>[ll]^-{\Xi_{\A\otimes\B}} \Un(\A)\otimes\Un(\B)
}
\end{align*}
Note that $\eta \cdot \Xi - \Id_{\A\otimes\B} =\mathfrak{M}^1(T)$ so $\eta \cdot \Xi \sim \Id_{\A\otimes\B}$ (so $\eta \cdot \Xi \approx \Id_{\A\otimes\B}$).\\

By construction the followings properties are immediate:
\begin{thm}\label{teoremono}
Let $\A_1$ be an $\Ain$-category, we have: 
\begin{itemize}
\item[1.] $\A_1\approx \Un(\A_1)$.
\item[2.] Given another $\Ain$-category $\A_2$, we have: 
\begin{align*}
\Un(\A_1\otimes\A_2)\approx \Un(\A_1)\otimes \Un(\A_2)
\end{align*}
\item[1.+ 2.]  Combining the previous items we have:
\begin{align*}
\Un(\A_1\otimes\A_2)\approx \A_1\otimes\A_2\approx \Un(\A_1)\otimes \Un(\A_2)
\end{align*}
for every pair of $\Ain$-categories $\A_1$ and $\A_2$.
\end{itemize}
\end{thm}

\begin{defn}\label{compofuncto}
Given two $\Ain$-functors $\F_1:\A_1\to\A_2$ and $\F_2:\A_3\to\A_4$ we can define:
\begin{align*}
\F_1``\otimes" \F_2:\A_1\otimes\A_3\to \A_2\otimes\A_4
\end{align*}
as the morphism fitting the diagram
\[
\xymatrix{
\A_1\otimes\A_3\ar@{-->}[rrr]\ar[d]_{\eta_{\A_1\otimes\A_3}}&&&\A_2\otimes \A_4\\
\Un(\A_1)\otimes \Un(\A_3)\ar[rrr]_{\Un(\F_1)\otimes^{\tiny\mbox{DG}} \Un(\F_2)}&&&\Un(\A_2)\otimes \Un(\A_4)\ar[u]_{\Xi_{\A_2\otimes\A_4}}
}
\]
In formula
\begin{align}\label{gronzo}
\F_1``\otimes" \F_2:=\Xi_{\A_2\otimes\A_4}\cdot \big( \Un(\F_1)\otimes^{\tiny\mbox{DG}} \Un(\F_2) \big)\cdot \eta_{\A_1\otimes\A_3}.
\end{align}
\end{defn}

\begin{thm}\label{trunchezfol}
Given two $\Ain$-categories $\B_1$, $\B_2$ and two $\Ain$-functors $\G_1:\B_1\to\B_2$ and $\G_2:\B_3\to\B_4$, we have:
\begin{align*}
(\F_1``\otimes"\G_1)\cdot (\F_2``\otimes"\G_2)\sim (\F_1\cdot\F_2)``\otimes" (\G_1\cdot\G_2).
\end{align*}
\end{thm}

\begin{proof}
It follows from the calculation:
\begin{align*}
(\F_1``\otimes"\G_1)\cdot(\F_2``\otimes"\G_2)&=\Xi_{\A_2\otimes\A_4}\cdot\big(\Un(\F_1)\otimes^{\tiny\mbox{DG}} \Un(\G_1)\big)\cdot\eta \cdot \Xi\cdot\big(\Un(\F_2)\otimes^{\tiny\mbox{DG}} \Un(\G_2)\big)\cdot\eta_{\A_1\otimes\A_3}\\
&\sim\Xi_{\A_2\otimes\A_4}\cdot\big((\Un(\F_1)\otimes^{\tiny\mbox{DG}} \Un(\G_1)) \cdot (\Un(\F_2)\otimes^{\tiny\mbox{DG}} \Un(\G_2))\big)\cdot\eta_{\A_1\otimes\A_3}\\
&=\Xi_{\A_2\otimes\A_4}\cdot\big(\Un(\F_1\cdot\F_2)\otimes^{\tiny\mbox{DG}} \Un(\G_1\cdot\G_2)\big)\cdot\eta_{\A_1\otimes\A_3}\\
&=(\F_1\cdot\F_2)``\otimes" (\G_1\cdot\G_2).
\end{align*}
\end{proof}

It follows immediately 

\begin{cor}\label{corzulo}
Given two $\Ain$-categories $\B_1$, $\B_2$ and two $\Ain$-functors $\G_1:\B_1\to\B_2$ and $\G_2:\B_3\to\B_4$, we have:
\begin{align*}
(\F_1``\otimes" \G_1)\cdot (\F_2``\otimes" \G_2)\approx (\F_1\cdot\F_2)``\otimes"  (\G_1\cdot\G_2).
\end{align*}
\end{cor}

\begin{lem}\label{drododo}
Let $\F\sim\G$ be two $\Ain$-functors and $\C$ an $\Ain$-category, 
then $\F\otimes\Id_{\C}\sim\G``\otimes" \Id_{\C}$ and $\Id_{\C}``\otimes"\F\sim\Id_{\C}``\otimes" \G$.
\end{lem}

\begin{proof}
We recall that we have the following commutative diagram
\begin{align*}
\xymatrix{
\A\ar[d]^{\F}\ar[r]^-{\alpha_{\A}}&\Un(\A)\ar[d]^{\Un(\F)}\\
\B\ar[r]^-{\alpha_{\B}}&\Un(\B)
}
\end{align*}
So $\Un(\F)\cdot\alpha_{\A}=\alpha_{\B}\cdot\F$ and $\Un(\F)\sim\alpha_{\B}\cdot\F\cdot\tilde{\tilde{\chi}}_{\A}$.\
It implies that, if $\F\sim\G$ then $\Un(\F)\sim\Un(\G)$.\ 
By Lemma \ref{stupidone} we have $\Un(\F)\otimes^{\tiny\mbox{DG}}\Id_{}\sim \Un(\G)\otimes^{\tiny\mbox{DG}}\Id$.\ 
It follows:
\begin{align*}
\F``\otimes"\Id_{\C}&:= \Xi\cdot \big( \Un(\F)\otimes^{\tiny\mbox{DG}}\Id_{\Un(\C)}\big) \cdot \eta\\
&\sim \Xi\cdot \big(\Un(\G)\otimes^{\tiny\mbox{DG}}\Id_{\Un(\C)}\big)\cdot \eta \\
&=\G``\otimes"\Id_{\C}.
\end{align*}
\end{proof}

\begin{lem}
Given an $\Ain$-category $\A$ we have $\alpha_{\Un(\A)}\sim\Un(\alpha_\A)$ and $\tilde{\tilde{\chi}}_{\Un(\A)}\sim\Un(\tilde{\tilde{\chi}}_\A)$.
\end{lem}

\begin{proof}
We have the following commutative diagram:
\begin{align*}
\xymatrix{
\A\ar[d]^{\alpha_\A}\ar[r]^-{\alpha_{\A}}&\Un(\A)\ar[d]^{\Un(\alpha_\A)}\\
\Un(\A)\ar[r]^-{\alpha_{\Un(\A)}}&\Un(\Un(\A))
}
\end{align*}
so $\Un(\alpha_\A)\cdot\alpha_{\A}=\alpha_{\Un(\A)}\cdot\alpha_\A$.\ 
Multiplying by $\tilde{\tilde{\chi}}_\A$, 
we have 
$\Un(\alpha_\A)\cdot\alpha_{\A}\cdot\tilde{\tilde{\chi}}_\A=\alpha_{\Un(\A)}\cdot\alpha_\A\cdot\tilde{\tilde{\chi}}_\A$ so
$\alpha_{\Un(\A)}\sim\Un(\alpha_\A)$.\\
Since 
$$\tilde{\tilde{\chi}}_{\Un(\A)}\cdot \alpha_{\Un(\A)}\sim\Un(\tilde{\tilde{\chi}}_{\A})\cdot\Un(\alpha_{\A})$$
we are done.
\end{proof}

By Lemma \ref{drododo} and Theorem \ref{trunchezfol} we have:

\begin{thm}
The tensor product defined in Theorem \ref{tenso} and in Theorem \ref{compofuncto} defines a bifunctor 
\begin{align*}
``\otimes" :\aCat/_{\sim}\mbox{ $\times$ } \aCat/_{\sim} \to \aCat/_{\sim}.
\end{align*}
\end{thm}

It is not difficult to prove that:
\begin{thm}
If $\F$ and $\G$ are cohomological (resp. strictly) unital $\Ain$-functors then $\F\otimes\G$ is a cohomological (resp. strictly) unital $\Ain$functor.
\end{thm}

Let $\A$ be a unital $\Ain$-category, we denote by $\Rep(\A)$, the (strictly unital) DG-category of representable $\Ain$-modules over $\A$ (see \cite[14.7]{BLM}, \cite[\S7, \S9]{Fuk} or \cite[(1j)]{Sei}).\ 
We have the following result:

\begin{thm}[Yoneda]\label{Yonni}
Given a unital $\Ain$-category $\A$:
\begin{itemize}
\item[1.]  There exists a unital $\Ain$-functor 
\begin{align*}
\REP:\A\to\Rep(\A)
\end{align*}
\item[2.] There exists a unital $\Ain$-functor
\begin{align*}
\Pi:\Rep(\A)\to\A
\end{align*}
\end{itemize}
such that $\REP\cdot\Pi\approx\Id_{\tiny\Rep(\A)}$ and $\Pi\cdot\REP\approx\Id_{\A}$.\\
Moreover, if $\A$ is strictly unital then 
\begin{itemize}
\item[2.] Fixed two objects $a$, $a'\in\A$, we can define the morphism of chain complexes:
\begin{align*}
\Pi_1:\Fun(\Rep(a),\Rep(a'))\to\HomA(a,a')
\end{align*}
as:
\begin{align}\label{Pi1}
\Pi_1(T):=T_0(a)(\id_a).
\end{align}
\end{itemize}
The chain maps $\REP_1$ and $\Pi_1$ satisfy 1.-2. of Theorem \ref{DC} and the corresponding $\Ain$-functors $\REP$ and $\Pi$ are strictly unital. 
\end{thm}

\begin{proof}
The case of unital case is \cite[1.5 Proposition]{BLM}.\
If $\A$ is strictly unital, the Definition of $\REP$ can be found in \cite{Orn3} and then the various check regarding $\Pi_1$ can be found in \cite[Definition 7.28 Lemma 9.8-9.21-9.22]{Fuk}.\ 
It remains to prove that, if $\A$ is strictly unital then $\Pi$ is strictly unital this check is left to the reader.
\end{proof}

\begin{lem}\label{storgheloz}
Let $\F\approx\G:\A\to\B$ be two $\Ain$-functors where $\B$ is a unital $\Ain$-category and $\C$ is a unital $\Ain$-category, 
then $\F``\otimes" \Id_{\C}\approx\G``\otimes" \Id_{\C}$ and $\Id_{\C}``\otimes" \F\approx\Id_{\C}``\otimes" \G$.
\end{lem}

\begin{proof}
Following the proof of Lemma \ref{drododo} we have that 
$\Un(\F)\cdot\alpha_{\A}=\alpha_{\B}\cdot\F$ 
and $\Un(\F)\sim\alpha_{\B}\cdot\F\cdot\chi_{\A}$.\
It implies that, if $\F\approx\G$ then $\Un(\F)\approx\Un(\G)$.\

Given a unital $\Ain$-category $\C''$, we have two unital $\Ain$-functors 
$\Rep:\C''\to\Rep(\C'')$ and $\Pi:\Rep(\C'')\to\C''$ (see Theorem \ref{Yonni})\
We denote by $\C'$ the DG category $\Rep(\C'')$.\
By Lemma \ref{stupidonzio} we have $\Un(\F)\otimes\Id_{\C'}\approx \Un(\G)\otimes\Id_{\C'}$.\ So
\begin{align*}
(\Un(\F)\otimes\Id_{\C'})\cdot (\Id_{\B}\otimes\Pi) \approx (\Un(\G)\otimes\Id_{\C'})\cdot(\Id_{\B}\otimes\Pi).
\end{align*}
By Corollary \ref{corzulo} we have 
\begin{align}\label{gorzull}
(\Un(\F) ``\otimes" \Pi) \approx (\Un(\G)``\otimes" \Pi).
\end{align}
Multiplying (\ref{gorzull}) by $(\Id_{\B}\otimes\Rep)$ we get:
\begin{align*}
(\Un(\F) \otimes\Id_{\C''}) \approx (\Un(\G)\otimes\Id_{\C''}).
\end{align*}
It follows:
\begin{align*}
\F``\otimes" \Id_{\C}&:= \Xi\cdot \big( \Un(\F)\otimes^{\tiny\mbox{DG}}\Id_{\Un(\C)}\big) \cdot \eta\\
&\approx \Xi\cdot \big(\Un(\G)\otimes^{\tiny\mbox{DG}}\Id_{\Un(\C)}\big)\cdot \eta \\
&=\G``\otimes" \Id_{\C}
\end{align*}
and we are done.
\end{proof}

By Lemma \ref{storgheloz} and Corollary \ref{corzulo} we have:

\begin{thm}
The tensor product defined in Theorem \ref{tenso} and in Definition \ref{compofuncto}, defines a bifunctor 
\begin{align*}
``\otimes":\aCat^{\star}/_{\approx}\mbox{ $\times$ } \aCat^{\star}/_{\approx} \to \aCat^{\star}/_{\approx}.
\end{align*}
Where $\star\in\mathcal{f}\mbox{su}, \mbox{ u} \mathcal{g}$ 
\end{thm}

\begin{rem}
We point out this important fact about our construction: given $\F$ and $\G$ two $\Ain$-functors, $U^{\tiny\mbox{n}}(\F)$ and $\Un(\G)$ are DG-functors.\
So, we can define the tensor product of two $\Ain$-functors $\F$ and $\G$ via the the product (as DG-functors) $\Un(\F)\otimes\Un(\G)$.\
Note that, to fulfill the hypothesis of the main construction (Theorem \ref{enzone}) we do not need any functoriality.\\
For example, we can provide a tensor product of two strictly unital $\Ain$-categories via Yoneda.\
Namely, by Theorem \ref{Yonni} and Theorem \ref{enzone} one can provide an $\Ain$-structure on the DG quiver $|\A|\otimes|\B|$.\

The downside of this construction is that, given an $\Ain$-functor $\F:\A\to\B$, the pushforward $\REP(\F):\Rep(\A)\to\Rep(\B)$ is an $\Ain$functor (not a DG functor).\
So we cannot define the tensor product of two $\Ain$-functors $\F\otimes\G$ as in Definition \ref{compofuncto}.\ 

On the other hand, it is possible to define the tensor product $\F\otimes\G$ as in \cite[Definition 3.9]{Amo} but, in this case, the relation between
$(\F\otimes\G)\cdot(\F'\otimes\G')$ and $\F\cdot\F'\otimes\G\cdot\G'$ is not clear (see \cite[pp 15]{Amo}).\ 

\end{rem}

\subsection{Symmetric monoidal structure}\label{sms}

In the previous subsection we gave the notion of tensor product $``\otimes"$ between two $\Ain$-functors.\\
We can define the tensor product $\otimes$ of two $\Ain$-functors $\F:\A\to\B$ and $\G:\A'\to\B'$ as follows:
\begin{align}\label{sanzo}
\F\otimes \G:=(\tilde{\tilde{\chi}}_{\A'}``\otimes"\tilde{\tilde{\chi}}_{\B'})\cdot(\Un(\F)\otimes^{\tiny\mbox{DG}}\Un(\G))\cdot (\alpha_\A``\otimes"\alpha_\B).
\end{align}

\begin{lem}\label{latrunk}
The tensor product $\otimes$ in equation (\ref{sanzo}) is well defined.\ 
\end{lem}

\begin{proof}
It suffices to prove that $\alpha_{\A}\otimes \alpha_\B\sim\alpha_{\A}``\otimes"\alpha_{\B}$ and 
$\chi_{\A}\otimes \chi_\B\sim\chi_{\A}``\otimes"\chi_{\B}$, for every $\Ain$-categories $\A$ and $\B$.\\
We calculate:
\begin{align*}
\alpha_{\A}\otimes \alpha_\B&:=(\tilde{\tilde{\chi}}_{\Un(\A)}``\otimes"\tilde{\tilde{\chi}}_{\Un(\B)})\cdot(\Un(\alpha_\A)\otimes^{\tiny\mbox{DG}}\Un(\alpha_\B))\cdot (\alpha_\A``\otimes"\alpha_\B)\\
&=(\tilde{\tilde{\chi}}_{\Un(\A)}``\otimes"\tilde{\tilde{\chi}}_{\Un(\B)})\cdot(\Un(\alpha_\A) ``\otimes"\Un(\alpha_\B))\cdot (\alpha_\A``\otimes"\alpha_\B)\\
&\sim (\tilde{\tilde{\chi}}_{\Un(\A)}``\otimes"\tilde{\tilde{\chi}}_{\Un(\B)})\cdot(\alpha_{\Un(\A)}``\otimes"\alpha_{\Un(\B)})\cdot (\alpha_\A``\otimes"\alpha_\B)\\
&\sim (\tilde{\tilde{\chi}}_{\Un(\A)}\cdot \alpha_{\Un(\A)})``\otimes"(\tilde{\tilde{\chi}}_{\Un(\B)}\cdot \alpha_{\Un(\B)})\cdot (\alpha_\A``\otimes"\alpha_\B)\\
&= \alpha_\A``\otimes"\alpha_\B.
\end{align*}
\end{proof}

The tensor product $\otimes$ between two $\Ain$-functors has the same properties of $``\otimes"$.\ Namely:

\begin{itemize}
\item If $\mathsf{F}:\mathcal{C}_1\to\mathcal{C}'_1$, $\mathsf{G}:\mathcal{C}_2\to\mathcal{C}'_2$ such that $\mathsf{F},\mathsf{G}\in\DgCat$ then 
$$\mathsf{F}\otimes^{\tiny\mbox{DG}}\mathsf{G}=\mathsf{F}\otimes\mathsf{G}.$$
\item If $\F\sim\G$ then $\F\otimes\Id\sim \G\otimes\Id$ and $\Id\otimes\F\sim\Id\otimes\G$.
\item If $\F:\A\to\B$ and $\G:\A'\to\B'$ then 
$$(\F\otimes\G)\cdot(\F'\otimes\G')\sim (\F\cdot\F')\otimes(\G\cdot\G').$$
\item If $\F\approx\G$ then $\F\otimes\Id\approx \G\otimes\Id$ and $\Id\otimes\F\approx\Id\otimes\G$.
\end{itemize}



\begin{thm}\label{gronzulcore}
The tensor product defined in Theorem \ref{tenso} and in equation (\ref{sanzo}) defines a bifunctor 
\begin{align*}
\otimes :\aCat^{\ddagger}/_{\sim}\mbox{ $\times$ } \aCat^{\ddagger}/_{\sim} \to \aCat^{\ddagger}/_{\sim}.
\end{align*}
and a bifunctor:
\begin{align*}
\otimes :\aCat^{\star}/_{\approx}\mbox{ $\times$ } \aCat^{\star}/_{\approx} \to \aCat^{\star}/_{\approx}.
\end{align*}
Where $\ddagger=\mathcal{f}\mbox{su, nu, u, cu}\mathcal{g}$ and ${\star}=\mathcal{f}\mbox{su, u}\mathcal{g}$.
\end{thm}

\begin{thm}\label{symcat}
The categories $(\aCat^{\ddagger}/\sim, \otimes,R)$ and $(\aCat^{\star}/\approx,\otimes,R)$ are symmetric monoidal, 
where $\ddagger=\mathcal{f}\mbox{su, nu, u, cu}\mathcal{g}$ and $\star=\mathcal{f}\mbox{su, u}\mathcal{g}$.
\end{thm}

\begin{proof}
We define the right-left unitor, the swap map and the associator as follows:
\begin{itemize}
\item[1.] Given an $\Ain$-category $\A_1$, the \emph{right unitor} is the strict $\Ain$-functor
\begin{align*}
_rI_{\A_1}:\A_1\otimes R&\cong \A_1\\
_rI_{\A_1}^1(c_1\otimes r_1\cdot 1_R)&:= r_1\cdot c_1
\end{align*}
the \emph{left unitor} is the strict $\Ain$-functor
\begin{align*}
_lI_{\A_1}:R\otimes\A_1&\cong \A_1\\
_lI_{\A_1}^1(r_1\cdot 1_R\otimes c_1)&:= r_1\cdot c_1.
\end{align*}
\item[2.] Given another $\Ain$-category $\A_2$, the \emph{swap map} 
\begin{align*}
\mu_{\A_1 \A_2}:\A_1\otimes \A_2&\cong \A_2\otimes\A_1
\end{align*}
is the equivalence of $\Ain$-categories defined as 
$$\mu_{\A_1 \A_2}:=(\tilde{\tilde{\chi}}_{\A_2}\otimes \tilde{\tilde{\chi}}_{\A_1})\cdot \mu^{\tiny\mbox{DG}}_{\Un(\A_1)\Un(\A_2)} \cdot (\alpha_{\A_1}\otimes\alpha_{\A_2}).$$
\item[3.] Let $\A_3$ be another $\Ain$-category, the \emph{associator} 
\begin{align*}
\sigma_{\A_1 \A_2 \A_3}:\A_1\otimes(\A_2 \otimes\A_3)&\cong (\A_1\otimes \A_2)\otimes\A_3
\end{align*}
is the equivalence of $\Ain$-categories, defined as
$$\sigma_{\A_1\A_2\A_3}:=\big((\tilde{\tilde{\chi}}_{\A_1}\otimes \tilde{\tilde{\chi}}_{\A_2})\otimes \tilde{\tilde{\chi}}_{\A_3}\big)\cdot\sigma^{\tiny\mbox{DG}}_{\Un(\A_1)\Un(\A_2)\Un(\A_3)}\cdot\big( \alpha_{\A_1}\otimes(\alpha_{\A_2}\otimes\alpha_{\A_3})\big).$$
\end{itemize}
The coherence conditions, the unit coherence, the associative coherence and the inverse law comes from directly from the coherence conditions of DG-categories.
\end{proof}




\newpage

\section{Derived tensor product and Internal Homs}

The tensor product $\otimes$ makes $\DgCat$ a symmetric monoidal category.\
This structure is closed and the internal hom is given by 
$$[\A,\B]:=\mbox{Nat}(\A,\B)$$
Where $\mbox{Nat}(\A,\B)$ denotes the DG-category of DG-functors between $\A$ and $\B$, whose morphisms are (the DG enriched) natural transformations, see \cite{Orn2}.\\
On the other hand, the category $\Ho(\DgCat)$ has an induced symmetric monoidal category.\ 
The derived tensor product is given by 
\begin{align*}
\A\otimes^{\mathbb{L}}\B:=\A\otimes\B^{\tiny\mbox{cof}}.
\end{align*}
Where $\B^{\tiny\mbox{cof}}$ is a cofibrant resolution of $\B$ in the model structure of $\DgCat$ (see \cite{Tab}).\\
Note that every cofibrant object in $\DgCat$ is h-projective morphisms (see \cite[Proposition 2.3]{Toe}).\\

In the case of the categories $\Ho(\DgCat^u)$, $\Ho(\aCat^u)$ and $\Ho(\aCat)$, 
we can define the \emph{derived tensor product} as follows:
\begin{align*}
\A\otimes^{\mathbb{L}}\B:=\A\otimes\B^{\tiny\mbox{hp}}
\end{align*}
where $\B^{\tiny\mbox{hp}}$ denotes a h-projective resolution of $\B$.

\begin{rem}
If $\A$ is a h-projective $\Ain$-category then $\Un(\A)$ is h-projective.\ Note that in general $\Rep(\A)$ (and $U(\A)$) is not h-projective.
\end{rem}

In this section we prove that the categories $(\Ho(\DgCat^{u}),\otimes^{\mathbb{L}},R)$ and 
$(\Ho(\aCat^{(u)}),\otimes^{\mathbb{L}},R)$
are closed symmetric monoidal category.\\ 
More precisely we prove the following:
\begin{thm}\label{INTERNALHOM}
Given $\A_1,\A_2,\A_3\in\star$, there exists a natural bijection of sets
\begin{align*}
{\mbox{Ho}(\star)}(\A_1\otimes^{\mathbb{L}}\A_2,\A_3)\simeq{\mbox{Ho}(\star)}(\A_1, \Fun^u_{\infty}(\A^{\tiny\mbox{hp}}_2,\A_3)),
\end{align*}
where $\A^{\tiny\mbox{hp}}_2$ is a h-projective resolution of $\A_2$ and $\star\in\mathcal{f}\DgCat,\DgCat^u,\aCat^u,\aCat\mathcal{g}$.\\
Moreover, if $\A_1,\A_2,\A_3\in\aCat$ and $\A_1\otimes^{\mathbb{L}}\A_2:=\A_1\otimes\tilde{\A}_2$, such that $\tilde{\A}_2$ has split unit, then 
\begin{align*}
{\mbox{Ho}(\aCat)}(\A_1\otimes^{\mathbb{L}}\A_2,\A_3)\simeq{\mbox{Ho}(\aCat)}(\A_1, \Fun_{\infty}(\A_2^{\tiny\mbox{hp}},\A_3)).
\end{align*}
\end{thm}

\subsection{Internal Hom}
We recall that the category of DG categories $\Ho(\DgCat)$ is closed symmetric monoidal.\ 
The internal hom of $\Ho(\DgCat)$ has a very explicit incarnation in terms of $\Ain$-functors:  
\begin{thm}[\cite{COS2}]\label{IntHom}
Given $\A_1,\A_2,\A_3\in\DgCat$, and $\A_2$ h-projective there exists a natural bijection of sets
\begin{align*}
\xi:{\mbox{Ho}(\DgCat)}(\A_1\otimes\A_2,\A_3)\simeq{\mbox{Ho}(\DgCat)}(\A_1, \Fun^u_{\infty}(\A_2,\A_3)).
\end{align*}
\end{thm}

We extend this result to the categories $\DgCat^u$, $\aCat^u$ and $\aCat$.\\
To give a sketch of proof of Theorem \ref{IntHom}, we recall that, given three unital $\Ain$-categories $\A$, $\B$, $\C$, then
$$\aCat^u(\A,\B,\C)$$ 
denotes the unital $\Ain$-category whose objects are the $\Ain$-bifunctors from $\A$ and $\B$ to $\C$.\\
A bifunctor $F\in\aCat^u(\A,\B,\C)$ is a DG-functor 
between the DG cocategories:
\begin{align*}
\overline{B_{\infty}(\A)_{+}\otimes B_{\infty}(\B)_{+}}\to B_{\infty}(\C)_{+}
\end{align*}
Note that $\aCat^u(\A,\B,\C)$ has the same structure of $\C$, namely if $\C\in\star$ then $\aCat^u(\A,\B,\C)\in\star$ (see \cite[\S2]{Orn2}).

\begin{proof}[(Sketch of) proof of Theorem \ref{IntHom}]
The isomorphism $\xi$ is given by the following natural isomorphism of sets (see \cite[(5.1)]{COS2}):
\begin{align}\label{diagrammozzo}
\xymatrix{
\Ho(\DgCat)(\A\otimes\B,\C)\ar[d]^{1:1}\ar@{-->}[rr]^{\xi}&&\ar@{-->}[ll]\Ho(\DgCat)(\A,\Fun_{\infty}^u(\B,\C))\ar[d]^{(F)}\\
\Ho(\DgCat)(\A^{\tiny\mbox{hp}}\otimes\B,\C)\ar[u]^{(A)}\ar[d]^{1:1}&&\Ho(\DgCat)(\A^{\tiny\mbox{hp}},\Fun_{\infty}^u(\B,\C))\ar[d]^{(E)}\ar[u]^{1:1}\\
\aCat^u(\A^{\tiny\mbox{hp}}\otimes\B,\C)/{\approx}\ar[u]^{(B)}\ar[r]^-{1:1}&\ar[l]^-{(C)}\aCat^u(\A^{\tiny\mbox{hp}},\B,\C)/{\approx}\ar[r]^-{1:1}&\ar[l]^-{(D)}\aCat^u(\A^{\tiny\mbox{hp}},\Fun^u_{\infty}(\B,\C))/\approx\ar[u]^-{1:1}
}
\end{align}
\end{proof}

\begin{lem}\label{isonzo}
Let $\A$, $\B$, $\C$ be three unital $\Ain$-categories we have an isomorphism of sets:
\begin{align}\label{iso1}
\aCat^u(\A,\B,\C)\simeq \aCat^u(\A,\Fun^u(\B,\C))
\end{align}
and
\begin{align}\label{iso2}
\aCat^u(\A,\B,\C)\simeq\aCat^u(\B,\A,\C).
\end{align}
\end{lem}

\begin{proof}
The isomorphism (\ref{iso1}) is a particular case of \cite[Proposition 5.1]{COS2} 
and (\ref{iso2}) follows directly from the definition.
\end{proof}

\begin{thm}\label{corzel}
Given $\A_1,\A_2,\A_3\in\DgCat^u$ and $\A_2$ h-projective, there exists a natural bijection of sets
\begin{align*}
{\mbox{Ho}(\DgCat^u)}(\A_1\otimes\A_2,\A_3)\simeq{\mbox{Ho}(\DgCat^u)}(\A_1, \Fun^u_{\infty}(\A_2,\A_3)),
\end{align*}
\end{thm}

To prove Theorem \ref{corzel} we need two more results.

\begin{lem}\label{tensoronzi}
Let $\A$, $\B$ and $\C$ be three unital DG-categories, we have an isomorphism of sets:
$$\aCat^u(\A\otimes\B,\C)/{\approx}\to \aCat^u(\A,\B,\C)/{\approx}$$
\end{lem}

\begin{proof}
The proof is the same of \cite[Proposition 5.2]{COS2}.\ 
The only thing to prove, in order to use \cite[Lemma 1.22]{COS2}, is that the category
$$\overline{\A_+\otimes\B_+}$$
is a unital DG-category, if $\A$ and $\B$ are unital DG-categories.\\
We denote by $e_a$ and $e_b$ the units of $a\in\A$ and $b\in\B$ satisfying (\ref{unitrap}) (with homotopies $\mathcal{H}_a$ and $\mathcal{H}_b$ respectively).\ Denoting by $\Id_a$ and $\Id_b$ the units added in $\A_+$ and $\B_+$,
we take the morphism
\begin{align}\label{unitor}
e_a\otimes\Id_b+\Id_a\otimes e_b-e_a\otimes e_b\in\Hom_{\overline{\A_+\otimes\B_+}}\big( (a,b),(a,b)\big).
\end{align}
This is the unit of the object $(a,b)\in\A\otimes\B$ we can take $\mathcal{H}_a\otimes(d_{\B}\mathcal{H}_b+\mathcal{H}_b d_{\B})$ satisfying (\ref{unitrap}).\ 
\end{proof}

We recall the following result \cite[Theorem 6.19]{Orn2}:

\begin{thm}\label{patturzo}
We have an morphism of sets:
\begin{align}\label{gollork}
\aCat^u(\A,\B)/{\approx}&\to \Ho(\aCat^u)(\A,\B)\\
\delta:\F&\mapsto[\F].
\end{align}
Moreover, if $\A$ is h-projective then $\delta$ is an isomorphism.
\end{thm}


\begin{proof}[Proof of Theorem \ref{corzel}]
We take (\ref{diagrammozzo}) replacing $\DgCat$ with $\DgCat^u$:
\begin{align}\label{diagrammozzone}
\xymatrix{
\Ho(\DgCat^u)(\A\otimes\B,\C)\ar[d]^{1:1}\ar@{-->}[rr]^{\xi'}&&\ar@{-->}[ll]\Ho(\DgCat^u)(\A,\Fun_{\infty}^u(\B,\C))\ar[d]^{(F)}\\
\Ho(\DgCat^u)(\A^{\tiny\mbox{hp}}\otimes\B,\C)\ar[u]^{(A)}\ar[d]^{1:1}&&\Ho(\DgCat^u)(\A^{\tiny\mbox{hp}},\Fun_{\infty}^u(\B,\C))\ar[d]^{(E')}\ar[u]^{1:1}\\
\aCat^u(\A^{\tiny\mbox{hp}}\otimes\B,\C)/{\approx}\ar[u]^{(B')}\ar[r]^-{1:1}&\ar[l]^-{(C')}\aCat^u(\A^{\tiny\mbox{hp}},\B,\C)/{\approx}\ar[r]^-{1:1}&\ar[l]^-{(D)}\aCat^u(\A^{\tiny\mbox{hp}},\Fun^u_{\infty}(\B,\C))/\approx\ar[u]^-{1:1}
}
\end{align}
The isomorphisms $(A)$, $(F)$ hold and $(D)$ comes from Lemma \ref{isonzo}.\\ 
Moreover Theorem \ref{equivalonze} and Theorem \ref{patturzo} provide the isomorphisms $(B')$/$(E')$ and the isomorphism $(C)$ is due to Lemma \ref{tensoronzi}.
\end{proof}

\begin{proof}[Proof of Theorem \ref{INTERNALHOM}]
Given $\A_1,\A_2,\A_3\in\aCat^u$, we have:
\begin{align*}
\Ho(\aCat^u)(\A_1\otimes^{\mathbb{L}}\A_2,\A_3)&:= \Ho(\aCat^u)(\A_1\otimes\A_2^{\tiny\mbox{hp}},\A_3)\\
&\simeq \Ho(\aCat^u)(\A_1^{\tiny\mbox{hp}}\otimes\A_2^{\tiny\mbox{hp}},\A_3)\\
&\simeq \Ho(\DgCat^u)(U^n(\A_1^{\tiny\mbox{hp}}\otimes\A_2^{\tiny\mbox{hp}}),U^n(\A_3))\\
&\simeq \Ho(\aCat^u)(U^n(\A_1^{\tiny\mbox{hp}}\otimes\A_2^{\tiny\mbox{hp}}),U^n(\A_3))\\
&\simeq \Ho(\aCat^u)(U^n(\A_1^{\tiny\mbox{hp}})\otimes U^n(\A_2^{\tiny\mbox{hp}}),U^n(\A_3))\\
&\simeq \Ho(\DgCat^u)(U^n(\A_1^{\tiny\mbox{hp}})\otimes U^n(\A_2^{\tiny\mbox{hp}}),U^n(\A_3))\\
&\simeq \Ho(\DgCat^u)\big(U^n(\A_1^{\tiny\mbox{hp}}), \Fun^u_{\infty}( U^n(\A_2^{\tiny\mbox{hp}}),U^n(\A_3))\big)\\
&\simeq \Ho(\DgCat^u)\big(U^n(\A_1^{\tiny\mbox{hp}}), U^n(\Fun^u_{\infty}( U^n(\A_2^{\tiny\mbox{hp}}),U^n(\A_3)))\big)\\
&\simeq \Ho(\aCat^u)\big(\A_1^{\tiny\mbox{hp}}, \Fun^u_{\infty}( U^n(\A_2^{\tiny\mbox{hp}}),U^n(\A_3))\big)\\
&\simeq \Ho(\aCat^u)\big(\A_1^{\tiny\mbox{hp}}, \Fun^u_{\infty}( \A_2^{\tiny\mbox{hp}},\A_3)\big)\\
&\simeq \Ho(\aCat^u)\big(\A_1, \Fun^u_{\infty}( \A_2^{\tiny\mbox{hp}},\A_3)\big)
\end{align*}
The same is true if $\A_1,\A_2,\A_3\in\aCat$.\ Moreover if $\A^{\tiny\mbox{hp}}_2$ has split unit then 
$$\Fun^u_{\infty}( \A_2^{\tiny\mbox{hp}},\A_3)\simeq \Fun_{\infty}( \A_2^{\tiny\mbox{hp}},\A_3)$$
and we are done.
\end{proof}

\begin{thm}\label{closedsm}
The symmetric monoidal category $(\aCat^{u}/\approx,\otimes,R)$ is closed.\ 
An internal hom $[\B,\C]$ is given by $\Fun^{u}(\B,\C)$.
\end{thm}

\begin{proof}
Since $\A\approx\Un(\A)$ and $\Un(\A)$ is a unital DG-category, it follows directly from Theorem \ref{teoremono}, Lemma \ref{isonzo} and Lemma \ref{tensoronzi}. 
\end{proof}

\newpage

\section{Explicit Formulas}\label{explicit}

Let $\A$ and $\B$ be two $\Ain$categories, in this section we calculate explicitly $m^n_{\A\otimes\B}$ for $n=1,2,3$ and we give a recursive formula for any $n\ge1$.\ See formulas (\ref{multer}) and (\ref{zonzo}).\\
By Theorem \ref{adjun} we have:
\begin{align*}
\alpha^1_{\A}\cdot\chi^1_{\A}-\mbox{Id}_{\Un(\A)}=d_{\Un(\A)}\cdot H_{\A}+H_{\A}\cdot d_{\Un(\A)}.
\end{align*}
Clearly the same is true for the category $\B$.\\
We recall that the $\Ain$-structure on $\A\otimes\B$ is given by Theorem \ref{enzone}.\\ 
In particular we use Theorem \ref{HPT} on the diagram:
\begin{align*}
\xymatrix{
\big({\A\otimes\B}\big)(x,y)\ar@<0.5ex>[rr]^-{\alpha^1_{\A}\otimes\alpha^1_{\B}}&&\ar@<0.5ex>[ll]^-{\chi^1_{\A}\otimes\chi^1_{\B}} \big({\Un(\A)\otimes \Un(\B)}\big)(x,y) \ar@(r,d)[]^{T^1}
}
\end{align*}
Where $T^1$ is given by Lemma \ref{stupido}:
\begin{align}\label{torcofo}
T^1:=H_{\B}\otimes \big(\alpha^1_{\A}\cdot\chi^1_{\A}\big)+\mbox{Id}_{\Un(\B)}\otimes H_{\A}.
\end{align}

\begin{rem} 
On the other hand we can take also the homotopy 
\begin{align*}
T_r^1:=H_{\B}\otimes \big(\alpha^1_{\A}\cdot\chi^1_{\A}\big)+\mbox{Id}_{\Un(\B)}\otimes H_{\A}.
\end{align*}
\end{rem}

\begin{notat}
We introduce a new notation in this section in order to make the exposition and the formulas (hopefully) lighter.\ 
We denote simply by $f$ the double shift $(f[1])[-1]$, we denote with the underline notation $\underline{f_n\otimes...\otimes f_1}$ the morphism $(f_n[1]\otimes...\otimes f_1[1])[-1]$.\ With this notation the morphism (\ref{gurzo}) is given by 
\begin{align*}
R(\underline{f_n\otimes...\otimes f_1}\otimes g):=0,\\
R(f_n\otimes g):= \underline{f_n\otimes g}.
\end{align*}
Where $g\in U(\A)$ and $f_j\in\A$ for every $j=1,...,n$.
\end{notat}

\begin{exmp}
Let us make a few of calculations:\\
If $f$ has length 1 we have: 
\begin{itemize}
\item[1.] 
$r(f)=f$,\\
$\chi^1_{\A}(f)=f$,\\
$R(f)=H_{\A}(f)=0$
\end{itemize}
If $f$ has length 2 then:
\begin{itemize}
\item[2a.] 
$\chi^1_{\A}(\underline{a_2\otimes a_1})=0$,\\ 
$H_{\A}(\underline{a_2\otimes a_1})=0$.\\
\item[2b.] $\chi^1_{\A}({a_2\otimes a_1})=m^2_{\A}(a_2,a_1)$,\\ 
$H_{\A}({a_2\otimes a_1})=\underline{a_2\otimes a_1}$.
\end{itemize}
If $f$ has length 3 then:
\begin{itemize}
\item[3a.] 
$r({a_3\otimes a_2\otimes a_1})= m^2(a_3,a_2)\otimes a_3$,\\
$\chi^1_{\A}({a_3\otimes a_2\otimes a_1})=m^2(m^2(a_3,a_2),a_1)$, \\
$H_{\A}({a_3\otimes a_2\otimes a_1})=\underline{a_3\otimes a_2}\otimes a_1+ \underline{m^2_{\A}(a_3,a_2)\otimes a_1}$.\\
\item[3b.] 
$r(\underline{a_3\otimes a_2}\otimes a_1)= \underline{m^2(a_3,a_2)\otimes a_1}$,\\
$\chi^1_{\A}(\underline{a_3\otimes a_2}\otimes a_1)=0$, \\
$H_{\A}(\underline{a_3\otimes a_2}\otimes a_1)=0$.\\
\item[3c.] 
$r(a_3\otimes\underline{a_2\otimes a_1})=\underline{m^2(a_3,a_2)\otimes a_1}+m^3(a_3,a_2,a_1)$,\\
$\chi^1_{\A}(a_3\otimes \underline{a_2\otimes a_1})=m^3(a_3,a_2,a_1)$,\\
$H_{\A}(a_3\otimes \underline{a_2\otimes a_1})=\underline{a_3\otimes a_2\otimes a_1}$.\\
\item[3d.] 
$r(\underline{a_3\otimes a_2\otimes a_1})=0$,\\
$\chi^1_{\A}( \underline{a_3\otimes a_2\otimes a_1})=0$, \\
$H_{\A}( \underline{a_3\otimes a_2\otimes a_1})=0$.
\end{itemize}
\end{exmp}

The $\Ain$functor $\eta_{\A\otimes\B}:\A\otimes\B\to \Un(\A)\otimes \Un(\B)$ is given recursively by the formula:
\begin{align}\label{zonzo}
\eta_{\A\otimes\B}^n(f_n,...,f_1):=&\displaystyle\sum_{j=1}^{n-1} (\mbox{Id}_{\A} \otimes H_{\B}) \big( m^2_{\Un(\A)\otimes \Un(\B)}(\eta_{\A\otimes\B}^{n-j}(f_{n},...,f_{j+1}), \eta_{\A\otimes\B}^{j}(f_{j},...,f_1))\big)\\
+&\displaystyle\sum_{j=1}^{n-1} (H_{\A}\otimes \alpha^1_{\B}\cdot\chi^1_{\B})\big( m^2_{\Un(\A)\otimes \Un(\B)}(\eta_{\A\otimes\B}^{n-j}(f_{n},...,f_{j+1})\otimes \eta_{\A\otimes\B}^{j}(f_{j},...,f_1)\big)\big).
\end{align}
On the other hand, the $\Ain$-structure on $\A\otimes\B$ is given recursively using the morphism $\eta_{\A\otimes\B}$:
\begin{align}\label{multer}
m^n_{\A\otimes\B}(f_n,...,f_1):=
&\displaystyle\sum_{j=1}^{n-1} \chi^1_{\A\otimes\B}\cdot m^2_{\Un(\A)\otimes \Un(\B)}\big(\eta_{\A\otimes\B}^{n-j}(f_{n},...,f_{j+1}),\eta_{\A\otimes\B}^{j}(f_{j},...,f_1)\big).
\end{align}

\begin{exmp}\label{mdue}
Let us calculate explicitly $m^2_{\A\otimes\B}$.\\
We recall:
\begin{align*}
\eta^1(a_1,b_1):=\big(\alpha^1_{\A}\otimes\alpha^1_{\B}\big)(a_1,b_1)=a_1\otimes b_1.
\end{align*}
Given $f_2:=a_2\otimes b_2$ and $f_1:=a_1\otimes b_1$, we have
\begin{align*}
m^2_{\A\otimes\B}(f_2,f_1):=
&\chi^1_{\A\otimes\B} \cdot m^2_{\Un(\A)\otimes \Un(\B)}\big( \alpha_{\A\otimes\B}^{1}(f_{2}),\alpha_{\A\otimes\B}^{1}(f_1)\big)\\
=&\chi^1_{\A\otimes\B} \cdot m^2_{\Un(\A)\otimes \Un(\B)}\big( a_{2}\otimes b_{2},a_{1}\otimes b_{1}\big)\\
=&\chi^1_{\A\otimes\B} \big( (a_{2}\otimes a_{1})\otimes (b_{2}\otimes b_{1})\big)\\
=&\chi^1_{\A}\otimes \chi^1_{\B} \big( (a_{2}\otimes a_{1})\otimes (b_{2}\otimes b_{1})\big)\\
=&m^2_{\A}(a_2,a_1)\otimes m^2_{\A}(b_2,b_1)
\end{align*}
\end{exmp}

\begin{exmp}\label{mtre}
Let us calculate explicitly $m^3_{\A\otimes\B}$.\\
First we have:
\begin{align*}
\eta^2(f_2,f_1)&:=(\mbox{Id}_{\A} \otimes H_{\B}) \big( (a_2\otimes a_1) \otimes (b_2\otimes b_1) \big)\\
+& (H_{\A} \otimes \alpha^1_{\B}\cdot\chi^1_{\B}) \big( (a_2\otimes a_1) \otimes (b_2\otimes b_1 \big)\\
=&(a_2\otimes a_1)\otimes\underline{b_2\otimes b_1}+ \underline{a_2\otimes a_1}\otimes m^2_{\B}(b_2,b_1).
\end{align*}
We take $f_j:=a_j\otimes b_j$ for $j=1,2,3$.\\
Let us calculate:
\begin{align*}
m^3_{\A\otimes\B}(f_3,f_2,f_1):=
&\chi^1_{\A\otimes\B}\big(m^2_{\Un(\A)\otimes \Un(\B)}(\eta^2(f_3,f_2),\eta^1_{\A\otimes\B}(f_1) )\big)\\ 
+&\chi^1_{\A\otimes\B}\big(m^2_{\Un(\A)\otimes \Un(\B)}(\eta^1_{\A\otimes\B}(f_3),\eta^2(f_2,f_1) )\big)\\ 
=&\chi^1_{\A\otimes\B}\big(m^2_{\Un(\A)\otimes \Un(\B)}\big( (a_3\otimes a_2)\otimes\underline{b_3\otimes b_2} + \underline{a_3\otimes a_2}\otimes m^2_{\B}(b_3,b_2),a_1\otimes b_1 \big)\\ 
+& \chi^1_{\A\otimes\B}\big(m^2_{\Un(\A)\otimes \Un(\B)}\big( a_3\otimes b_3, (a_2\otimes a_1)\otimes\underline{b_2\otimes b_1} +\underline{a_2\otimes a_1}\otimes m^2_{\B}(b_2,b_1) \big)\\
=&\chi^1_{\A\otimes\B}\big( (a_3\otimes a_2\otimes a_1)\otimes(\underline{b_3\otimes b_2}\otimes b_1) + (\underline{a_3\otimes a_2}\otimes a_1)\otimes (m^2_{\B}(b_3,b_2)\otimes b_1) \big)\\ 
+& \chi^1_{\A\otimes\B}\big(  (a_3\otimes a_2\otimes a_1)\otimes (b_3\otimes\underline{b_2\otimes b_1} )+ (a_3\otimes\underline{a_2\otimes a_1})\otimes (b_3\otimes m^2_{\B}(b_2,b_1)) \big)\\
=&m^2_{\A}(m^2_{\A}(a_3,a_2),a_1)\otimes m^3_{\B}(b_3,b_2,b_1) + m^3(a_3,a_2,a_1)\otimes m^2_{\B}(b_3,m^2_{\B}(b_2,b_1)).
\end{align*}
\end{exmp}

\newpage
\section{Appendix}\label{appendix}

\begin{proof}
We want to prove that $R$ verifies (\ref{lollz}).\\
If $n_m>1$ then $R\big(g_m[-1]\otimes...\otimes g_1[-1] \big)=0$, it is easy to see: 
\begin{align*}
R\big(d_{\Un(|\A|)}\big(g_m[-1]\otimes...\otimes g_1[-1]\big)\big)&=g_m[-1]\otimes...\otimes g_1[-1].
\end{align*}
We take the case $m=2$ and $n_2=1$.\\ 
In formula, we have a morphism of the form:
$$(f[1])[-1]\otimes (f_{n}[1]\otimes...\otimes f_1[1])[-1]$$
we have:
$$R\big((f[1])[-1]\otimes (f_{n}[1]\otimes...\otimes f_1[1])[-1]\big)=(f[1]\otimes f_{n}[1]\otimes...\otimes f_1[1])[-1].$$
The differential is
\begin{align*}
d_{\Un(|\A|)}\big((f[1])[-1]\otimes &(f_{n}[1]\otimes...\otimes f_1[1])[-1]\big)=(m_1(f)[1])[-1]\otimes (f_{n}[1]\otimes...\otimes f_1[1])[-1]+\\
&+(f[1])[-1]\otimes d_{\Un(|\A|)}\big((f_{n}[1]\otimes...\otimes f_1[1])[-1]\big)\\
&=(m_1(f)[1])[-1]\otimes (f_{n}[1]\otimes...\otimes f_1[1])[-1]+\\
&+\displaystyle\sum_{j=1}^n (f[1])[-1]\otimes (f_{n}[1]\otimes...\otimes m_1(f_j)[1]\otimes...\otimes f_1[1])[-1]\\
&+\displaystyle\sum_{j=2}^n (f[1])[-1]\otimes (f_{n}[1]\otimes...\otimes f_j[1])[-1]\otimes(f_{j-1}[1]\otimes...\otimes f_1[1])[-1].
\end{align*}
Applying $R$ we get:
\begin{align*}
R\big(d_{U^{\tiny\mbox{n}}(|{\mathscr{A}}|)}\big((f[1])[-1]\otimes& (f_{n}[1]\otimes...\otimes f_1[1])[-1]\big)\big)=(m_1(f)[1]\otimes f_{n}[1]\otimes...\otimes f_1[1])[-1]+\\
&+\displaystyle\sum_{j=1}^n (f[1]\otimes f_{n}[1]\otimes...\otimes m_1(f_j)[1]\otimes...\otimes f_1[1])[-1]\\
&+\displaystyle\sum_{j=2}^n (f[1]\otimes f_{n}[1]\otimes...\otimes f_j[1])[-1]\otimes(f_{j-1}[1]\otimes...\otimes f_1[1])[-1].
\end{align*}
On the other hand we calculate 
\begin{align*}
d_{\Un(|\A|)}\big(R((f[1])[-1]\otimes (f_{n}[1]&\otimes...\otimes f_1[1])[-1]) \big)=d_{U(\overline{A})}\big((f[1]\otimes f_{n}[1]\otimes...\otimes f_1[1])[-1]) \big)\\
&=(f[1])[-1]\otimes (f_{n}[1]\otimes...\otimes f_1[1])[-1])\\
&+\displaystyle\sum_{j=2}^n (f[1]\otimes f_{n}[1]\otimes...\otimes f_j[1])[-1]\otimes(f_{j-1}[1]\otimes...\otimes f_1[1])[-1]\\
&+(m_1(f)[1] \otimes f_{n}[1]\otimes...\otimes f_1[1])[-1]))\\
&+\displaystyle\sum_{j=1}^n (f[1]\otimes f_{n}[1]\otimes...\otimes m_1(f_j)[1]\otimes...\otimes f_1[1])[-1].
\end{align*}
We have 
\begin{align*}
R\big(d_{\Un(|\A|)}\big((f[1])[-1]\otimes (f_{n}[1]\otimes...&\otimes f_1[1])[-1]\big) \big)=\\
&=d_{\Un(|\A|)}\big( R \big((f[1])[-1]\otimes (f_{n}[1]\otimes...\otimes f_1[1])[-1]\big)\big)\\
&-(f[1])[-1]\otimes (f_{n}[1]\otimes...\otimes f_1[1])[-1]).
\end{align*}
So (\ref{lollz}) hold if $f=(f[1])[-1]\otimes (f_{n}[1]\otimes...\otimes f_1[1])[-1]$.\\
To prove the general case i.e. $m>2$ and $n_m=1$ we note that 
\begin{align*}
R\big(g_m[-1]\otimes g_{m-1}[-1]\otimes g_{m-2}[-1]\otimes...\otimes g_1[-1]\big)&=R\big(g_m[-1]\otimes g_{m-1}[-1]\big)\otimes g_{m-2}[-1]\otimes...\otimes g_1[-1].
\end{align*}
We calculate:
\begin{align*}
d_{\Un(|\A|)}\big(R\big(g_m[-1]\otimes g_{m-1}[-1]&\otimes g_{m-2}[-1]\otimes...\otimes g_1[-1]\big)\big)=\\
&=d_{\Un(|\A|)}\big(R\big(g_m[-1]\otimes g_{m-1}[-1]\big)\big)\otimes g_{m-2}[-1]\otimes...\otimes g_1[-1]+\\
&+R\big(g_m[-1]\otimes g_{m-1}[-1]\big)\otimes d_{\Un(|\A|)}\big(g_{m-2}[-1]\otimes...\otimes g_1[-1]\big).\\
&=g_m[-1]\otimes g_{m-1}[-1]\otimes g_{m-2}[-1]\otimes...\otimes g_1[-1]+\\
&+R\big(d_{\Un(|\A|)}\big(g_m[-1]\otimes g_{m-1}[-1]\big)\big)\otimes g_{m-2}[-1]\otimes...\otimes g_1[-1]+\\
&+R\big(g_m[-1]\otimes g_{m-1}[-1]\big)\otimes d_{\Un(|\A|)}\big(g_{m-2}[-1]\otimes...\otimes g_1[-1]\big)\\
&=g_m[-1]\otimes g_{m-1}[-1]\otimes g_{m-2}[-1]\otimes...\otimes g_1[-1]+\\
&+R\big(d_{\Un(|\A|)}\big(g_m[-1]\otimes g_{m-1}[-1]\big)\otimes g_{m-2}[-1]\otimes...\otimes g_1[-1]+\\
&+g_m[-1]\otimes g_{m-1}[-1]\otimes d_{U^{\tiny\mbox{n}}(\tilde{\mathscr{A}})_{+}}\big(g_{m-2}[-1]\otimes...\otimes g_1[-1]\big)\big)\\
&=g_m[-1]\otimes g_{m-1}[-1]\otimes g_{m-2}[-1]\otimes...\otimes g_1[-1]+\\
&+R\big(d_{\Un(|\A|)}\big(g_m[-1]\otimes g_{m-1}[-1]\otimes g_{m-2}[-1]\otimes...\otimes g_1[-1]\big)\big)\\
\end{align*}
and we are done.
\end{proof}

\end{document}